\documentclass[11pt, twoside]{article}
\usepackage{amsfonts,amssymb,amsmath,amsthm}
\usepackage{graphicx}
\usepackage[all]{xy}
\usepackage{tikz}
\usepackage{changepage}
\usepackage{psfrag,xmpmulti,amscd,color,pstricks, import}
\usepackage{amsmath,amsfonts,amssymb,graphics,mathrsfs,amsthm,eurosym,verbatim,enumerate,subcaption, enumitem}

 \divide\oddsidemargin by 2
\newtheorem{thm}{Theorem}[section]

\newtheorem{lem}[thm]{Lemma}

\theoremstyle{definition}

\newtheorem{rem}[thm]{Remark}
\newtheorem*{rem*}{Remark}
\newtheorem{ex}{Example}

\newtheorem{obs}{Observation}

\numberwithin{equation}{section}

\definecolor{OrangeRed}{cmyk}{0,0.6,1,0}            
\definecolor{DarkBlue}{cmyk}{1,1,0,0.20}
\definecolor{DarkGreen}{cmyk}{1,0,0.6,0.2}
\definecolor{myblue}{rgb}{0.66,0.78,1.00}
\definecolor{Violet}{cmyk}{0.79,0.88,0,0}
\definecolor{Lavender}{cmyk}{0,0.48,0,0}

\renewcommand{\red}{\color{black}}
\makeatletter
\def\blfootnote{\gdef\@thefnmark{}\@footnotetext}
\makeatother

\newcommand{\diam}{\operatorname{diam}}

\newcommand{\dist}{\operatorname{dist}}

\newcommand{\C}{{\mathbb C}}

\newcommand{\D}{{\mathbb D}}

\newcommand{\N}{{\mathbb N}}

\newcommand{\R}{{\mathbb R}}

\newcommand{\lran}{\underset{n\to\infty}{\longrightarrow}}

\renewcommand{\phi}{\varphi}
\newcommand{\eps}{\varepsilon}

\renewcommand{\tilde}{\widetilde}

\vspace{5cm}

\title{Oscillating simply connected wandering domains}

\author{Vasiliki Evdoridou, Philip J. Rippon and Gwyneth M. Stallard}

\date{}
\begin{document}
\maketitle
\begin{abstract}
Although detailed descriptions of the possible types of behaviour inside periodic Fatou components have been known for over 100 years, a classification of wandering domains has only recently been given. Recently, simply connected wandering domains were classified into nine possible types and examples of {\it escaping} wandering domains of each of these types were constructed. Here we consider the case of {\it oscillating} wandering domains, for which only six of these types are possible. {\red We use a new technique based on approximation theory to construct examples of all six types of oscillating simply connected wandering domains.} This requires delicate arguments since oscillating wandering domains return infinitely often to a bounded part of the plane. Our technique is inspired by that used by Eremenko and Lyubich to construct the first example of an oscillating wandering domain, but with considerable refinements which enable us to show that the wandering domains are bounded, to specify the degree of the mappings between wandering domains and to give precise descriptions of the dynamical behaviour of these mappings.
\end{abstract}

\section{Introduction}
\blfootnote{\hspace*{-0.6cm}All authors were supported by Engineering and Physical Sciences Research Council grant EP/R010560/1.\\2010 Mathematics Subject Classification. Primary 37F10; Secondary 30D05.}
Let $f$ be a transcendental entire function. We consider the iterates of $f$, which we denote by $f^n$, $n \geq 1$. The complex plane is divided into two sets: the Fatou set, $F(f)$, where the iterates $(f^n)$ form a normal family in a neighbourhood of every point, and its complement, the Julia set $J(f)$. An introduction to the theory of iteration of transcendental entire and meromorphic functions can be found in \cite{bergweiler93}.

The Fatou set is open and consists of connected components, which are called \textit{Fatou components}. Fatou components can be periodic, preperiodic or wandering domains. A Fatou component $U$ is called a \textit{wandering domain} if $f^n(U) \cap f^m(U)= \emptyset,$ for all $n\neq m$. Although Sullivan showed in \cite{sullivan} that rational maps have no wandering domains, transcendental entire functions can have wandering domains. The first example of such a function was given by Baker \cite{Baker76} who proved that a {\red certain} entire function given by an infinite product has a multiply connected wandering domain. Several examples of simply connected wandering domains have been constructed since then (see, for example,~\cite[p. 104]{Herman}, \cite[p. 414]{sullivan}, \cite[p. 564, p. 567]{Baker-wd}, \cite[p. 222]{Devaney-entire}, \cite[Examples 1 and 2]{pathex}, \cite{fagella-henriksen}).

In \cite{BRS} the authors gave a complete description of the dynamical behaviour in multiply connected wandering domains. Recently, in \cite{BEFRS} the authors gave a detailed classification of simply connected wandering domains in terms of the hyperbolic distance between orbits of points and in terms of convergence to the boundary. More specifically, they classified simply connected wandering domains into \textit{contracting}, \textit{semi-contracting} and \textit{eventually isometric} depending on whether, for almost all pairs of points in the wandering domain, the hyperbolic distances between the orbits of these points, tend to 0, decrease but do not tend to 0, or are eventually constant, respectively. In terms of convergence to the boundary, the orbits of all points \textit{stay away} from the boundary, come arbitrarily close to the boundary but do not converge to it (\textit{bungee}), or \textit{converge} to the boundary. These two classifications give nine possible types of simply connected wandering domains. Using a new technique, based on approximation theory, they show that all of these nine possible types are indeed realisable.

All the examples constructed in \cite{BEFRS} were {\it escaping} wandering domains. Hence it is natural to ask whether there exist {\it oscillating} wandering domains of all nine types. (It remains a major open question as to whether it is possible to have wandering domains of bounded orbit.) We first recall that {\red in oscillating wandering domains the iterates of~$f$ have finite limit points, as well as $\infty$, so it is impossible for the orbit of a point in such a wandering domain to stay away from the boundary.} Thus three of the nine possible types are not realisable.  In this paper we show that the remaining six possible types of oscillating wandering domains are all realisable.

The first transcendental entire function with oscillating wandering domains was given by Eremenko and Lyubich in \cite{pathex}; this was also the first application of approximation theory in complex dynamics. The authors used sequences of discs and half-annuli and a model function which was constant on the half-annuli and a translation on the discs. This model function was approximated on the closure of every disc and half-annulus by a transcendental entire function using an extended version of Runge's approximation theorem. {\red Their technique did not show though whether their wandering domains are bounded or not, and did not give information on the degree of the entire function on each of the wandering domains.}

Motivated by the construction in \cite{pathex}, we adapt the new techniques from \cite{BEFRS} to construct bounded oscillating wandering domains, which, moreover, have the property that the degree of~$f$ on each of the wandering domains is equal to that of our model map. We then use this technique to construct the six {\red possible} types of such wandering domains.

We prove our main result Theorem~\ref{thm:main construction} in Section~3. It is worth pointing out that, in order for the wandering domains to be oscillating, the set up needs to be much more complicated than that used for escaping wandering domains in \cite{BEFRS}. Although some of the building blocks of our proof are similar to those used in the analogous result for escaping wandering domains, the proof here requires several additional techniques. In particular, great care has to be taken over the accumulating errors in the approximation, as each of the discs $D_n$ on which the approximation takes place contains infinitely many domains in the orbit of the wandering domain. {\red Throughout, $D(z,r)$ denotes the open disc with centre~$z$ and radius~$r$.}

\begin{thm}[Main construction]\label{thm:main construction}
 Let $(b_{n})_{n\geq0}$ be {a sequence of} Blaschke products of {corresponding} degree $d_n \geq 1$, and let $(\alpha_n)_{n\geq 0}$ be a sequence of real numbers with $\alpha_0 = 1$ and $\alpha_{n+1}/\alpha_n \leq 1/6$. For $n \geq 0$, let
 \[
 D_n=D(9n,\alpha_n),\]
 \[ \Delta_n=D(a_n,\alpha_{n}) \;\mbox{ and } \Delta'_n=D(a_n,2\alpha_{n}), \mbox{ where } a_n=9n+4\alpha_{n},\]
and
 \[ G_n=D(\kappa_n, 1) \;\mbox{ and }\;\; G_n'= D(\kappa_n,5/4), \mbox{ where }  \kappa_n=a_n+3.
 \]
 We consider the function
$$
\phi(z) =
\begin{cases}
z+9,\;\;\mbox{ if } z\in \overline{D_n}, \; n \geq 0,\\
\frac{z-a_n}{\alpha_n}+\kappa_n, \;\; \mbox{ if } z \in \overline{\Delta_n'},\; n \geq 0,\\
\alpha_{n+1} b_n({z-\kappa_n})+4\alpha_{n+1}, \;\;\mbox{ if }z \in \overline{G_n'},\; n \geq 0,
\end{cases}$$
   and the sets
   \[
 V_m= D(\zeta_m, \rho_m) = {\red \phi^m(\Delta_0)=}
\begin{cases}
\Delta_n, \; \mbox{ if } m=\ell_n -1, n \geq 0,\\
G_n,\; \mbox{ if } m=\ell_n, n \geq 0,\\
D(9k + 4\alpha_{n+1},\alpha_{n+1}) \subset D_k, \; \mbox{ if } m = \ell_n + k+1,   0 \leq k \leq n,
\end{cases}
\]
where $(\ell_n)$ is defined by $\ell_0=1 \mbox{ and } \ell_{n+1}=\ell_n+n+3, \; n \geq 0.$
\begin{figure}[hbt!]\label{Fig1}
\centering
\begin{tikzpicture}[scale=2] 




\draw[black] (-4.7,0) circle [radius=0.6];
\draw[black] (-3.1,0) circle [radius=0.6];
\draw[black] (-1.5,0) circle [radius=0.6];
\draw[gray](-4.3,0) circle [radius=0.15];
\draw[gray](-4.6,0) circle [radius=0.06];
\draw[black] (-0.5,0) circle [radius=0.15];
\draw[black] (0,0) circle [radius=0.15];
\draw[black] (1, 0) circle [radius=0.6];
\draw [->] (-2,0.2) to [out=130,in=60] (-4.35,0);
\draw [->] (-4.35,-0.15) to [out=295,in=230] (0,-0.2);
\draw [->] (-2.65,-0.15) to [out=310,in=230] (-2,-0.1);
\draw [->] (0.05,-0.1) to [out=310,in=230] (0.55,-0.1);
\draw [->] (0.7,0.2) to [out=140,in=50] (-4.55,0.1);
\draw[fill] (-4.7,0) circle [radius=0.01];
\draw[fill] (-1.5,0) circle [radius=0.01];
\draw[fill] (-3.1,0) circle [radius=0.01];
\draw[fill] (1,0) circle [radius=0.01];

\node at (-4.7,-0.1) {$0$};
\node at (-3.1,-0.1) {$a_0=4$};
\node at (-1.5,-0.1) {$\kappa_0=7$};
\node at (1.05,-0.1) {$\kappa_1=a_1+3$};
\node at (-4.74,-0.75) {$D_0$};
\node at (-0.5,-0.3) {$D_1$};
\node at (0,0.25) {$\Delta_1=V_3$};
\node at (-4.3,0.25) {$V_2$};
\node at (-4.65,0.15) {$V_5$};
\node at (-3.1,0.35) {$\Delta_0=V_0$};
\node at (-1.45,0.75) {$G_0=V_{\ell_0}=V_1$};
\node at (0.95,0.75) {$G_1=V_{\ell_1}=V_4$};
\end{tikzpicture}
\vspace*{-0.5cm}
\caption{The action of the {\red model function $\phi$, discussed further in Section~3.}} 
\end{figure}
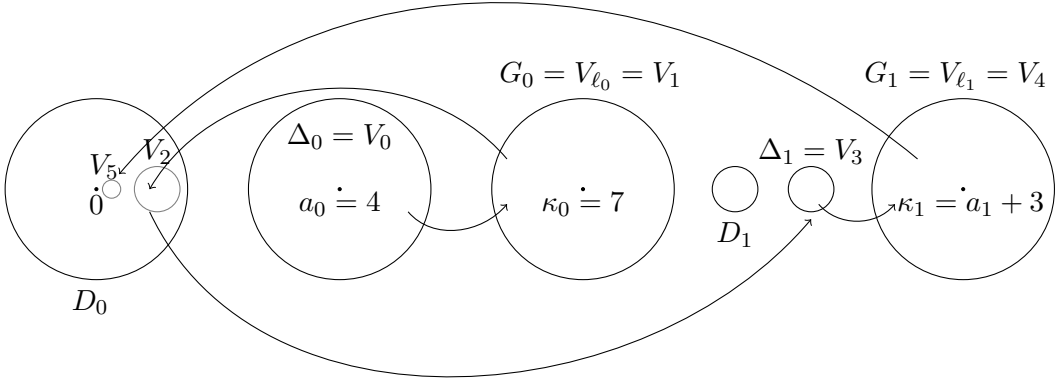

For a suitable choice of $(\alpha_n)$, there exists a transcendental entire function~$f$ having an orbit of bounded, simply connected, oscillating, wandering domains $U_m$
such that, for $m,n\geq 0$,
\begin{itemize}
\item[(i)]  $\overline{D(\zeta_m,r_m)} \subset U_m \subset D(\zeta_m, R_m)$, where $0<r_m<\rho_m<R_m$, and $r_m\sim \rho_m$ and $R_m \sim \rho_m$ as $m \to \infty$;
\item[(ii)] $|f(z)-\phi(z)|\leq \varepsilon_m$ on $\overline{D(\zeta_m, R_m)}$, where $\varepsilon_0 \leq 1/24$ and $\varepsilon_{\ell_n+k} = \frac{\alpha_{n+1}^2}{2^{k+1}}$, for $0\leq k \le n+2$;
\item[(iii)] $f(9n)=\phi(9n)=9(n+1)$ and $f'(9n)=\phi'(9n)= 1$;
\item[(iv)] $f:U_{m} \to U_{m+1}$ has degree $q_{m},$ where $q_{\ell_n}=d_n$, and $q_m=1$ otherwise.
\end{itemize}

 Finally, if $z, z' \in U_0$  and there exists $N \in \N$ such that $f^{\ell_N}(z), f^{\ell_N}(z') \in \overline{D(\kappa_N,r_{\ell_N})}$, then, for $n \geq N$, we have
\begin{equation}\label{eqtn:double inequality}
 k_n\operatorname{dist}_{G_n}(f^{\ell_n}(z), f^{\ell_n}(z')) \leq \operatorname{dist}_{U_{\ell_n}}(f^{\ell_n}(z), f^{\ell_n}(z'))\leq  K_n \operatorname{dist}_{G_n}(f^{\ell_n}(z), f^{\ell_n}(z')),
\end{equation}
where  $0<k_n<1<K_n$ with $k_n,K_n \to 1$ as $n \to \infty$.
\end{thm}

\begin{rem}
If the Blaschke products $b_n$ are real-symmetric for each $n \geq 0$, then $f$ can be taken to be real-symmetric; see Remark~\ref{rem2}.
\end{rem}

In Section 5 we use Theorem \ref{thm:main construction} to construct all six types of oscillating wandering domains, proving the following result. This requires several preliminary results concerning Blaschke products which we prove in Section 4.

\begin{thm}\label{realizable}
For each of the six possible types of simply connected oscillating wandering domains, there exists a transcendental entire function with a bounded, simply connected oscillating wandering domain of that type.
\end{thm}

Oscillating wandering domains for functions in the Eremenko-Lyubich class $\mathcal{B}$ have been constructed, first by Bishop in \cite{bishop15}, using the novel technique of quasiconformal folding, and more recently in \cite{FJL} and \cite{marshi}. It would be interesting to see whether their methods can be adapted to classify the resulting wandering domains as one of the six possible types described above.

\subsection*{Acknowledgments}
We would like to thank Anna Miriam Benini, Chris Bishop, Nuria Fagella and Lasse Rempe for inspiring discussions.
\section{Preliminary results for Theorem \ref{thm:main construction}}
 In this section we give some existing results which are used in the proof of Theorem \ref{thm:main construction}. The following theorem, which is \cite[Theorem D]{BEFRS}, plays a key role in the proof.

\begin{thm}\label{WDexist}
Let $f$ be a transcendental entire function and suppose that there exist Jordan curves $\gamma_n$ and $\Gamma_n$, $n\ge 0$, compact sets $L_k$, $k\ge 0$, and a bounded domain $D$ such that
\begin{itemize}
\item[\rm(a)] $\Gamma_n$ surrounds $\gamma_n$, for $n \geq 0$;
\item[\rm(b)] the sets $\Gamma_n$, $n \geq 0$, $L_k$, $k \geq 0$, and $\overline D$ all lie exterior to each other;
\item[\rm (c)]  $\gamma_{n+1}$ surrounds $f(\gamma_n)$, for $n \geq 0$;
\item[\rm (d)] $f(\Gamma_n)$ surrounds $\Gamma_{n+1}$, for $n \geq 0$;
\item[\rm (e)] $f(\overline D \cup \bigcup_{k\ge 0} L_k)\subset D$;
\item[\rm (f)] there exists $n_k \to \infty$ such that
$$
\max\{\operatorname{dist}(z,L_{k}): z \in \Gamma_{n_k}\} = o(\operatorname{dist}(\gamma_{n_k}, \Gamma_{n_k}))\;\text{ as}\; k \to \infty.
$$
\end{itemize}
Then there exists an orbit of simply connected wandering domains $U_n$ 
such that $\overline{\operatorname{int} \gamma_n} \subset U_n \subset \operatorname{int}\Gamma_n$, for $n \geq 0$.

Moreover, if there exists $z_n \in \operatorname{int}\gamma_n$ such that both  $f(\gamma_n)$ and $f(\Gamma_n)$ wind $d_n$ times around $f(z_n),$  then $f:U_n \to U_{n+1}$ has degree $d_n$, for $n \geq 0$.
\end{thm}
In order to obtain the transcendental entire function with the required properties, we consider an analytic function which is our model function and then apply the following result which is an extension of the well-known Runge's approximation theorem and was the Main Lemma in~\cite{pathex}.

\begin{lem}\label{Runge E-L}
Let $(E_n)$ be a sequence of compact subsets of $\C$ with the following properties:
\begin{itemize}
\item[(i)] $\C \setminus E_n$ is connected, for $n \geq 0$;
\item[(ii)] $E_n \cap E_m = \emptyset$, for $n\neq  m$;
\item[(iii)] $\min\{|z|: z\in E_n\} \to \infty$ as $n \to \infty$.
\end{itemize}
Suppose $\psi$ is holomorphic on $E =\bigcup_{n=0}^{\infty} E_n$. For $n \geq 0$, let $\varepsilon_n>0$ and let
$z_{n} \in E_n$. Then there exists an entire function $f$ satisfying, for $n \geq 0$,
\begin{equation}
|f(z)-\psi(z)|<\varepsilon_n, \quad \text{for } z\in E_n;
\end{equation}
\begin{equation}
f(z_{n})=\psi(z_{n}), \quad f'(z_{n}) = \psi'(z_{n}).
\end{equation}
\end{lem}

\begin{rem}\label{rem2}
We note that if the sets $E_n$ are each real-symmetric (that is, $\overline{E_n}=E_n$), the function~$\psi$ is real-symmetric in~$E$ (that is, $\overline{\psi(\overline{z})}=\psi (z)$, for $z \in E$), and the points $z_{n}, n\geq 0$, are all real, then we can take the entire function~$f$ to be real-symmetric on $\C$. Indeed, if~$f$ satisfies the conclusions of Lemma \ref{Runge E-L}, then $g(z)=\frac{1}{2}(f(z)+\overline{f(\overline{z})})$ is real-symmetric and entire, and satisfies the conclusions of Lemma \ref{Runge E-L}.
\end{rem}

We also need the following result, which is a version of \cite[Lemma 2]{pathex}.

\begin{lem}
\label{lem:EL2}
Let $g$ be an analytic function in the disc $\{z:|z|<R\}$ such that $g(0)=g'(0)=0$ and $|g(z)|<\epsilon R$ for $|z|<R$ and some $\epsilon< 1/4$. Then
$$|g(z)| \leq \frac{\epsilon}{R}|z|^2,\;\text{ for }|z| <R.$$
\end{lem}

Finally, we need the following lemma about hyperbolic distances in discs, which is \cite[Lemma 5.2]{BEFRS}.
\begin{lem}
\label{lem:hyp estimate}
Suppose that $0<s<r<1<R$ and set
\[
c(s,R)= \frac{1-s^2}{R-s^2/R},\quad D_r=D(0,r) \quad\text{and}\quad D_R=D(0,R).
\]
If $|z|,|w|\leq s$, then
\begin{equation}
\label{hyp est 1}
\operatorname{dist}_{D_R}(z,w)= \operatorname{dist}_{\mathbb{D}}({z}/{R},{w}/{R})\geq c(s,R)\operatorname{dist}_{\mathbb{D}}(z,w),
\end{equation}
and
\begin{equation}
\label{hyp est 2}
\operatorname{dist}_{D_r}(z,w)= \operatorname{dist}_{\mathbb{D}}({z}/{r},{w}/{r})\leq \frac{1}{c(s/r,1/r)}\operatorname{dist}_{\mathbb{D}}(z,w).
\end{equation}
Also, $0<c(s,R)<1$ and if the variables~$s$,~$r$ and~$R$ satisfy in addition
\begin{equation}\label{srR}
1-r=o(1-s)\;\text{ as } s\to 1\quad\text{and}\quad R-1=O(1-r)\;\text{ as } r\to 1,
\end{equation}
then
\begin{equation}
\label{cR}
c(s,R)\to 1\;\text{ as}\;s \to 1,
\end{equation}
and
\begin{equation}\label{cr}
c\left(s/r,1/r \right) \to 1\;\text{ as}\;s \to 1.
\end{equation}
\end{lem}

\section{Proof of Theorem \ref{thm:main construction}}

In this section, we prove our construction result. We consider the sets $V_m = \phi^m(\Delta_0)$, where $\Delta_0 = D(4,1)$, as defined in the statement of Theorem~\ref{thm:main construction} and construct a function~$f$ which is sufficiently close to $\phi$ in parts of the plane in order to ensure that $f$ has a {\red bounded} wandering domain~$U$ with $f^m(U)$  close to $V_m$, {\red for $m\ge 0$}, in the sense that the Hausdorff distance between $U_m$ and $V_m$ tends to $0$ as $m \to \infty$.

{\red\subsection*{The sets $V_m$}}

We begin by noting that it follows from the definition of $\phi$ and the fact that $\alpha_{m+1}/\alpha_m \leq 1/6$, for $m \geq 0$, that, for each $n \geq 0$,
\[
\phi(\Delta_n)= G_n,
\]
\[
\phi^2(\Delta_n) = D(4\alpha_{n+1}, \alpha_{n+1}) \subset D(0,\alpha_0) = D_0,
\]
so, for $0 \leq k \leq n$,
\[
\phi^{k+2}(\Delta_n) = D(9k + 4\alpha_{n+1}, \alpha_{n+1}) \subset D(9k, \alpha_k) = D_k,
\]
and
\[
\phi^{n+3}(\Delta_n) = D(9(n+1) + 4\alpha_{n+1}, \alpha_{n+1}) = \Delta_{n+1}.
\]
{\red We obtain the following properties of $V_m$, as stated in Theorem~\ref{thm:main construction}, and illustrated in Figure 1:
\begin{equation}\label{V}
V_m = D(\zeta_m,\rho_m)=\phi^m(\Delta_0) =
\begin{cases}
\Delta_n, \;\; \mbox{ if } m=\ell_n -1, \; n \geq 0,\\
G_n,\;\; \mbox{ if } m=\ell_n, \; n \geq 0,\\
D(9k + 4\alpha_{n+1},\alpha_{n+1}) \subset D_k, \;\; \mbox{ if } m = \ell_n + k+1, \; 0 \leq k \leq n,
\end{cases}
\end{equation}
where $(\ell_n)$ is defined by $\ell_0 = 1 \mbox{ and } \ell_{n+1} = \ell_{n} + n+3,$ for $n\ge 0$.}

{\red In words}, if $V_m \subset D_0$, then $\phi$ repeatedly translates $V_m$ to the right by 9 until the translated image lands on $\Delta_{n}$, for some $n \in \N$, at which point $\phi$ maps the disc $\Delta_{n}$ onto $G_n$ and then maps $G_n$ into $D_0$ (see Figure 1).

\subsection*{Construction of the circles $\gamma_n$ and $\Gamma_n$}

We now give an inductive definition of the values $r_m$ and $R_m$ described in Theorem~\ref{thm:main construction}, part~(i), and define $\alpha_n$ inductively at the same time. We will choose these values in such a way that, if we define
\[
\gamma_m = \{z: |z-\zeta_m| = r_m\} \quad\mbox{and}\quad \Gamma_m = \{z: |z-\zeta_m| = R_m\},
\]
then, for $m \geq 0$,
\begin{equation}\label{propa}
\gamma_{m+1}\;\text{ surrounds}\;\phi(\gamma_m),
\end{equation}
and
\begin{equation}\label{propb}
\phi(\Gamma_m)\;\text{ surrounds}\; \Gamma_{m+1}.
\end{equation}
Further, we choose these values in such a way that we are able to use Lemma~\ref{Runge E-L} and Lemma~\ref{lem:EL2} to approximate the map $\phi$ by an entire function $f$ such that $\phi$ can be replaced by $f$ in~\eqref{propa} and~\eqref{propb}. This in turn allows us to apply Theorem~\ref{WDexist} to deduce that $f$ has wandering domains with the required properties.

Our construction uses the Blaschke products $b_n$ which, for $n \geq 0$, we write as
\[
b_n(z) = e^{i \theta_n} \prod_{j=1}^{d_n}\frac{z+p_{n,j}}{1+\overline{p_{n,j}}z},
\]
where $p_{n,j} \in \mathbb{D}=\{z:|z|<1\}$ are not necessarily different from each other, and $\theta_n \in [0, 2\pi)$. We also use the maps defined by
\begin{equation}\label{Bdef}
B_n(z) = b_n(z-\kappa_n), \mbox{ for } n\geq 0,
\end{equation}
{\red noting that $B_n$ has degree~$d_n$ and maps $G_n$ to $D_0=\D$.}

First take
\begin{equation}\label{r0}
{\red r_0 \in (5/6,1) \mbox{ and } R_0 \in (1,7/6),}
\end{equation}
and {\red recall that} $\alpha_0 = 1$. We then choose $r_1$ such that
\begin{equation}\label{r1}
0<1- r_1 \leq \min\left\{\frac{1- r_0}{2}, \operatorname{dist}(\phi(\gamma_0), \partial G_0)^2\right\}
\end{equation}
\[
B_0(\gamma_1)\; \text{winds exactly}\;d_0\;\text{times round}\;D(0,1/2),
\]
so
\[
\phi(\gamma_1)\; \text{winds exactly}\;d_0\;\text{times round}\;D(\zeta_2,\rho_2/2),
\]
and choose $R_1$ such that
\begin{equation}
\label{R1}
0< R_1 - 1 \leq \min \left\{\frac{R_0-1}{2}, {\red \operatorname{dist}(\phi(\Gamma_0), \partial G_0)}, \frac{1}{\max_j{|p_0,j}|-1}\right\}.
\end{equation}

Now assume that, for some $n \geq 0$, $\alpha_{k}$ has been chosen for $0 \leq k \leq n$, and $r_m$ and $R_m$ have been chosen for $0\leq m\leq \ell_{n}$. (Note that $\ell_0 = 1$ and we have already specified $\alpha_0$, $r_0$, $R_0$, $r_1$ and $R_1$.) We shall give a rule for choosing $\alpha_{n+1}$ and also for choosing $r_m$ and $R_m$ for $\ell_{n} + 1 \leq m \leq \ell_{n+1}$. There are three different cases depending on the value of $m$.\\

{\bf Case 1} \;First we consider the case when $m = \ell_{n} + 1$ (and so $V_m \subset D_0$). We also specify $\alpha_{n+1}$ as part of this case.

We begin by choosing
$c_{n+1}$, $C_{n+1}$ to be circles centred at 0, lying in the interior and exterior of $D_0$ respectively, such that
\begin{equation}\label{eta}
\operatorname{dist}(c_{n+1}, \partial D_0) \leq \min\left\{ \frac{\rho_{\ell_{n}}-r_{\ell_{n}}}{6}, \frac{1}{2} \operatorname{dist}(B_{n}(\gamma_{\ell_{n}}), \partial D_0)\right\}
\end{equation}
and
\begin{equation}\label{H}
\operatorname{dist}(C_{n+1}, \partial D_0) \leq \min \left\{\frac{R_{\ell_{n}}-\rho_{\ell_{n}}}{6}, \operatorname{dist}(c_{n+1}, \partial D_0), \frac{1}{2} \operatorname{dist}(B_{n}(\Gamma_{\ell_{n}}), \partial D_0)\right\}.
\end{equation}
We set
\begin{equation}
\label{eq:errors Lem2}
\alpha_{n+1} = \operatorname{dist}(C_{n+1},\partial D_0)
\end{equation}
and note, {\red using the fact that $\phi(z)=\alpha_{n+1}B_n(z)+4\alpha_{n+1}$, for $z \in G'_n$,} that $V_{\ell_n+1}=\phi(V_{\ell_n})=D(4\alpha_{n+1},\alpha_{n+1})$, so $\rho_{\ell_n+1} = \alpha_{n+1}$. We then set
\begin{equation}\label{r-2a}
 r_{\ell_{n}+1}= \rho_{\ell_{n}+1}- \alpha^2_{n+1}
\end{equation}
 and
\begin{equation}\label{R-2a}
 R_{\ell_{n}+1}= \rho_{\ell_{n}+1}+ \alpha^2_{n+1}.
\end{equation}
Note that, together with~\eqref{eta}, \eqref{eq:errors Lem2} and \eqref{H}, these definitions imply that

\begin{equation}\label{r-2}
 \rho_{\ell_{n}+1}- r_{\ell_{n}+1}  = \alpha^2_{n+1}\leq \alpha_{n+1} \operatorname{dist}(c_{n+1}, \partial D_0) \leq \frac{1}{2}\operatorname{dist}(\phi(\gamma_{\ell_{n}}), \partial V_{\ell_n +1}),
\end{equation}
\begin{equation}\label{R-2}
 R_{\ell_{n}+1} - \rho_{\ell_{n}+1}  = \alpha^2_{n+1} =  \alpha_{n+1} \operatorname{dist}(C_{n+1}, \partial D_0) \leq  \frac{1}{2}\operatorname{dist}(\phi(\Gamma_{\ell_{n}}), \partial V_{\ell_n +1}),
 \end{equation}
{\red and hence
\begin{equation}\label{R-r}
 R_{\ell_{n}+1} - r_{\ell_{n}+1} =2\alpha_{n+1}^2 \le \min\{\operatorname{dist}(\phi(\gamma_{\ell_{n}}), \partial V_{\ell_n +1}),\operatorname{dist}(\phi(\Gamma_{\ell_{n}}), \partial V_{\ell_n +1})\}.
 \end{equation}}

{\bf Case 2}\; We now consider the cases when $ m = \ell_{n} + k + 1$, for $1 \leq k \leq n+1$. Then
{\red
\[
V_{\ell_{n} + k + 1} = \phi^k(V_{\ell_n+1}) \subset D_{k}, \text{ for } 1\le k\le n,
\]
and
\[
V_{\ell_{n}+n+2} = V_{\ell_{n+1}-1} = D(9(n+1)+4\alpha_{n+1},\alpha_{n+1})=\Delta_{n+1}.
\]}
In all these cases, we simply choose $r_m$ and $R_m$ to satisfy
\begin{equation}\label{r-3}
 \rho_m-r_{m}= \frac{\rho_{m-1}-r_{m-1}}{2},
\end{equation}
and
\begin{equation}\label{R-3}
R_m-\rho_m= \frac{R_{m-1}-\rho_{m-1}}{2}.
 \end{equation}

{\bf Case 3} \;Finally, we consider the case when $m = \ell_{n+1} = \ell_n + n+3$, so $V_m = G_{n+1}$. In this case, we choose $r_{\ell_{n+1}}$ and $R_{\ell_{n+1}}$ so that

\begin{equation}
\label{r-4}
0< \rho_{\ell_{n+1}}-r_{\ell_{n+1}} \leq \min\left\{ \frac{\rho_{{\ell_{n+1}}-1}-r_{{\ell_{n+1}}-1}}{2}, \operatorname{dist}(\phi(\gamma_{\ell_{n+1}-1},\partial G_{n+1})^2\right\};
\end{equation}
\begin{equation}
\label{degree-4a}
{\red B_{n+1}(\gamma_{{\ell_{n+1}}})}\; \text{winds exactly}\;d_{n+1}\;\text{times round}\;D(0,1/2);
\end{equation}
\begin{equation}
\label{R-4}
0<R_{\ell_{n+1}} - \rho_{\ell_{n+1}} \leq \min \left\{\frac{R_{{\ell_{n+1}}-1}-\rho_{{\ell_{n+1}}-1}}{2},  \operatorname{dist}({\red \phi(\Gamma_{\ell_{n+1}-1}}, \partial G_{n+1}), \frac{1}{\max_j\{|p_{n+1,j}|\}}- 1\right\}.
\end{equation}

This inductive process defines the values $r_m$ and $R_m$, and hence the circles $\gamma_m$ and $\Gamma_m$, for $m \geq 0$. Note that it follows from \eqref{H},~\eqref{eq:errors Lem2}, \eqref{R-3}, {\red \eqref{R-2a}} and \eqref{R-4} that
\[
\alpha_{n+1} \leq \frac{R_{\ell_{n}} - \rho_{\ell_{n}}}{6} < \frac{R_{\ell_{n-1}+1} - \rho_{\ell_{n-1}+1}}{6} = \frac{\alpha^2_{n}}{6} <  \frac{\alpha_{n}}{6},\;\text{ for } n\ge 1.
\]

Moreover, it follows from the definition of $\phi$ together with \eqref{Bdef} and  \eqref{V} that we have

\begin{equation}\label{phidelta}
\phi(z) = \alpha_{n+1}B_n(z) + 4\alpha_{n+1} = \rho_{\ell_{n}+1}B_n(z) + {\red \zeta_{\ell_{n}+1}}, \mbox{ for } z \in G_n, n\geq 0.
\end{equation}
So~\eqref{degree-4a} implies that, for $n \geq 0$,

\begin{equation}
\label{degree-4}
\phi(\gamma_{{\ell_{n+1}}})\; \text{winds exactly}\;d_{n+1}\;\text{times round}\;D(\zeta_{\ell_{n+1}+1},\rho_{\ell_{n+1}+1}/2).
\end{equation}

We also note that it follows from~\eqref{R-2},~\eqref{R-3} and \eqref{V} that, for $m = \ell_{n}+k + 1$, where $0 \leq k \leq n+1$, we have {\red $R_m - \rho_m \le R_{\ell_n+1}- \rho_{\ell_n+1} < \alpha_{n+1} = \rho_m$.} So, for $m = \ell_{n}+k + 1$, where $0 \leq k \leq n$, we have
\begin{equation}\label{V'}
V_m' = D(\zeta_m,R_m) \subset  D(\zeta_m,2\rho_m) =  D(9k + 4\alpha_{n+1},2 \alpha_{n+1})  \subset D(9k, 6\alpha_{n+1}) \subset D_{k},
\end{equation}
and
{\red
\begin{equation}\label{Delta'}
V'_{\ell_{n}+n+2} = V'_{\ell_{n+1}-1} =  D(\zeta_{\ell_{n+1}-1}, R_{\ell_{n+1}-1}) \subset \Delta_{n+1}'.
\end{equation}
by the definitions of $D(\zeta_m,\rho_m)$ and $\Delta_{n}'$ in the statement of Theorem~\ref{thm:main construction}.}

It then follows from~\eqref{R-4} and~\eqref{V} that $R_{\ell_{n+1}} - \rho_{\ell_{n+1}} < \alpha_{n+1}$ and so
\begin{equation}\label{G'}
V'_{\ell_{n+1}} = D(\zeta_{\ell_{n+1}}, R_{\ell_{n+1}}) \subset G_{n+1}'.
\end{equation}
Also, $\phi$ is analytic in $V'_{\ell_{n+1}}$ by~\eqref{phidelta} together with the last condition in~\eqref{R-4}.

It follows from~\eqref{r-2} that, for $n \geq 0$,
\begin{equation}
\gamma_{\ell_n + 1} \mbox{ surrounds } \phi(\gamma_{\ell_n}),
\end{equation}
and from~\eqref{R-2} that, for $n \geq 0$,
\begin{equation}
\phi(\Gamma_{\ell_n}) \mbox{ surrounds } \Gamma_{\ell_n + 1}.
\end{equation}
Thus \eqref{propa} and \eqref{propb} hold when $m = \ell_n$, where $n \geq 0$.

Also, if $m = \ell_n+k+1$, where $n \geq 0$, $0 \leq k \leq n+1$, then $\phi$ is a translation on $\gamma_m$ and $\Gamma_m$, by~\eqref{V'} {\red and the definition of~$\phi$}. Since, by~\eqref{r-3}, we have
\[
\rho_{m+1}-r_{m+1} \leq \frac{\rho_{m}-r_{m}}{2}
\]
and, by~\eqref{R-3},
\[
R_{m+1}-\rho_{m+1} \leq \frac{R_m-\rho_{m}}{2},
\]
it follows that \eqref{propa} and \eqref{propb} hold for these values of~$m$ too.

Finally, it follows from~\eqref{Delta'} that, on $\gamma_{\ell_{n+1}-1}$ and $\Gamma_{\ell_{n+1}-1}$, {\red $n\ge 0$}, the function $\phi$ is a scaling by a factor of $1/\alpha_{n+1} > 1$ followed by a translation, and so it follows from~\eqref{r-4} and~\eqref{R-4} that \eqref{propa} and \eqref{propb} hold in this case too.

We note that the sets $\overline{V_m'}$ are disjoint since, if $V'_{\ell_n+k+1}, V'_{\ell_{n+1}+k+1} \subset D_k$, for some $n \geq 0$, {\red $0\le k\leq n$}, then $V'_{\ell_n+k+1} \subset D(9k + 4\alpha_{n+1}, 2\alpha_{n+1})$ and $V'_{\ell_{n+1}+k+1} \subset D(9k + 4\alpha_{n+2}, 2\alpha_{n+2})$, and
\[
D(9k + 4\alpha_{n+1}, 2\alpha_{n+1}) \cap D(9k + 4\alpha_{n+2}, 2\alpha_{n+2}) = \emptyset,
\]
since
\[4\alpha_{n+2} + 2\alpha_{n+2} = 6\alpha_{n+2} \leq \alpha_{n+1} < 4\alpha_{n+1} - 2\alpha_{n+1}.
\]


\subsection*{Construction of the function $f$}

Our aim now is to use Lemma \ref{Runge E-L} and Lemma~\ref{lem:EL2} to approximate the map $\phi$ by a single entire function~$f$ such that, for $m \geq 0$, $\gamma_{m+1}$ surrounds $f(\gamma_m)$ and $f(\Gamma_m)$ surrounds $\Gamma_{m+1}$. We also require~$f$ to map certain curves $L_n$ near $\overline{G_n}$ in such a way that we can apply Theorem~\ref{WDexist}.

We define $L_n$, for $n\geq 0$, to be the circular arc
 \begin{equation}\label{eq:reefs}
 L_n:=\{z:|z-a_n|=R_{\ell_n}+\delta_{\ell_n}^2/2,\;|\operatorname{arg}(z-a_n)| \leq \pi - \delta_{\ell_n}^2\},
 \end{equation}
where $\delta_m= R_m-r_m \to 0$ as $m \to \infty$; see Figure 2.
\begin{figure}[hbt!]
\centering
\begin{tikzpicture}[scale=2.2] 
\draw[blue] (2.2,0) circle [radius=0.11]; 

\draw  (-1.5,0) circle [radius=0.5]; 
\draw[blue]  (-1.5,0) circle [radius=0.7]; 
\draw[fill=lime] (-4.7,0) circle [radius=0.6];
\draw[fill=lime] (-3.1,0) circle [radius=0.6];
\draw[fill=lime] (-1.5,0) circle [radius=0.6];
\draw[fill=lime] (0.4,0) circle [radius=0.17];
\draw[gray](-4.3,0) circle [radius=0.15];
\draw[gray](-4.6,0) circle [radius=0.06];
\draw[fill=lime] (0.9,0) circle [radius=0.17];
\draw[fill=lime] (2,0) circle [radius=0.6];
\draw [->] (-1.9,0.2) to [out=130,in=50] (-4.3,0.2);
\draw [->] (1.8,0.2)to [out=140,in=50] (-4.6,0.08);
\draw [->] (-2.7,-0.15) to [out=310,in=220] (-1.9,-0.15);
\draw [->] (0.9,-0.15) to [out=310,in=220] (1.6,-0.15);
\draw[fill] (-3.1,0) circle [radius=0.01];
\draw[fill] (-0.3,0) circle [radius=0.01];
\draw[fill] (-0.2,0) circle [radius=0.01];
\draw[fill] (-4.7,0) circle [radius=0.01];
\draw[fill] (-0.1,0) circle [radius=0.01];
\draw[fill] (0.9,0) circle [radius=0.01];
\draw[fill] (2,0) circle [radius=0.01];
\draw[fill] (-1.5,0) circle [radius=0.01];
\draw[blue]  (-1.5,0) circle [radius=0.3]; 
\draw[blue]  (-3.1,0) circle [radius=0.72];
\draw[blue]  (-3.1,0) circle [radius=0.27];
\draw[blue] (0.9,0) circle [radius=0.22]; 
\draw[blue] (0.9,0) circle [radius=0.12];
\draw[blue] (2,0) circle [radius=0.65]; 
\draw[blue] (2,0) circle [radius=0.55];
\draw[red] (-0.7,0) arc (0:172:0.8);
\draw[red] (-1.5+0.8*cos{188},0.8*sin{188}) arc(188:360:0.8);
\draw[red] (2.7,0) arc (0:176:0.7);
\draw[red] (2+0.7*cos{188},0.7*sin{188}) arc(188:360:0.7);
\node at (-3.1,-0.1) {$4$};
\node at (-1.5,-0.1) {$7$};
\node at (-4.7,-0.1) {$0$};
\node at (2,-0.06) {$\kappa_n=a_n+3$};
\node at (0.9,-0.05) {$a_n$};
\node at (1.2,-0.45) {$\phi$};
\node at (-4.4,-0.75) {$D_0$};
\node at (-3.1,0.4) {$\Delta_0=V_0$};
\node at (-1.4,0.9) {$L_0$};
\node at (2.1,0.3) {$G_n=V_{\ell_n}$};
\node at (2.1,0.8) {$L_n$};
\node at (-1.5,0.4) {$G_0=V_1$};
\node at (0.4,-0.3) {$D_n$};
\node at (0.9,0.3) {$\Delta_n$};
\node at (-2.3,-0.5) {$\phi$};
\node at (-3.1,0.9) {$\phi$};
\node at (-1.3,1.6) {$\phi$};
\end{tikzpicture}
\caption{Sketch of the setup of Theorem \ref{thm:main construction}, {\red showing the location of the circles $\gamma_n$ and $\Gamma_n$ (in blue), and the arcs~$L_n$ (in red).}}
\end{figure}
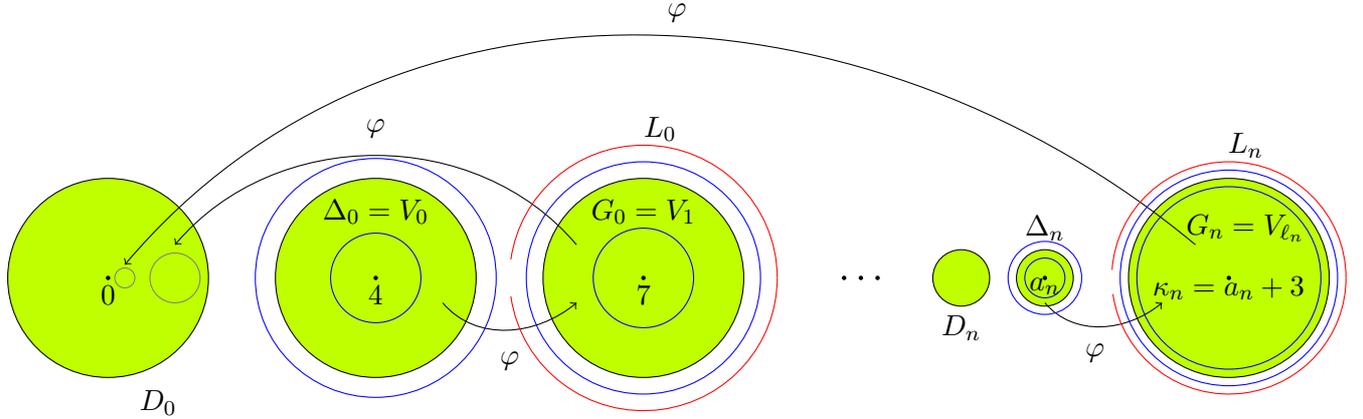

 We also define approximation error quantities $\varepsilon_m$, for $m \geq 0$, by
\begin{equation}\label{eq:error}
\varepsilon_m = \min \{\tfrac{1}{4}\operatorname{dist}(\phi(\gamma_{m}), \partial V_{m+1}), \tfrac{1}{4} \operatorname{dist}(\phi(\Gamma_{m}), \partial V_{m+1}), \tfrac14\delta_{m+1}\}>0.
\end{equation}


We now show that these errors have the upper bounds stated in part~(ii) of our theorem.
First, it follows from \eqref{r0}, \eqref{r1} and \eqref{R1} that
\[
{\red \varepsilon_0 \leq  (R_1 - r_1)/4 \leq (R_0 - r_0)/8 < 1/24.}
\]

{\red Next we note that it follows from~\eqref{R-r} that, for $n \geq 0$,
\[
\varepsilon_{\ell_n} = \delta_{\ell_n+1}/4=(R_{\ell_n +1} - r_{\ell_n +1})/4 = 2\alpha_{n+1}^2/4 = \alpha_{n+1}^2/2.
\]

It then follows from (\ref{r-3}), (\ref{R-3}), (\ref{r-4}) and (\ref{R-4}) that, for $0\leq k \leq n+1, \; n\geq 0$,
\begin{equation}\label{eq:error-translation}
\varepsilon_{\ell_n+k+1} = (R_{\ell_n +k + 2} - r_{\ell_n + k +2})/4  = (R_{\ell_n +1} - r_{\ell_n +1})/2^{k+3} = \alpha_{n+1}^2/2^{k+2}.
\end{equation}
Thus
\begin{equation}\label{eps}
\varepsilon_{\ell_n+k} = \alpha_{n+1}^2/2^{k+1},\; \mbox{ for } 0\leq k \leq n+2, \; n\geq 0,
\end{equation}
as required for part~(ii).}

Since $\phi$ is analytic in each set $\overline{V'_{\ell_{n+1}}}$, for $n \geq 0$, it follows from Lemma \ref{Runge E-L} that there exists an entire function~$f$ such that, for  $n \geq 0$,
\begin{equation}\label{approx1}
|f(z)-\phi(z)| < \varepsilon_{\ell_n + n+1}, \text{\ \  for }  z \in \overline{D_n},
\end{equation}
\begin{equation}
\label{approxnew}
|f(z)-\phi(z)| < \varepsilon_{\ell_n-1}, \text{\ \  for }  z \in \overline{\Delta_n'},
\end{equation}
\begin{equation}\label{approx2}
|f(z)-\phi(z)| < \varepsilon_{\ell_n}, \text{\ \  for }  z \in \overline{G_n'},
\end{equation}
\begin{equation} \label{approx5}
f(9n)=9({n+1}),
\end{equation}
\begin{equation}\label{approx6}
f'(9n)= 1,
\end{equation}
and such that
\begin{equation}\label{approx3}
|f(z)+4| \leq 1/2,  \text{\ \  for }  z\in \overline{D(-4,1)}  \cup \bigcup_{n\geq 0}  L_n.
\end{equation}

It follows from \eqref{approx1}, {\red \eqref{approx5} and \eqref{approx6}} that for each $k \geq 0$ we can apply Lemma \ref{lem:EL2} in the disc $D_k = D(9k, \alpha_k)$,  with $g(z)= f(z)-\phi (z)$, $R=\alpha_k$ and associated constant $\epsilon = \varepsilon_{\ell_k + k+1}/\alpha_{k}$. Note that the conditions of Lemma~\ref{lem:EL2} are satisfied since it follows from~\eqref{eps} that
\[
\epsilon = \varepsilon_{\ell_k + k+1}/\alpha_{k}  = \frac{\alpha_{k+1}^2}{\alpha_{k}2^{k+2}} < \frac{\alpha_{k+1}}{6}< 1/4, \mbox{ for } k\geq 0.
\]
So, by Lemma~ \ref{lem:EL2}, for all $z \in D_k$, $k \geq 0$, we have
\begin{equation}\label{fDk}
|f(z)-\phi(z)| \leq \frac{\varepsilon_{\ell_k + k+1}}{\alpha^2_{k}}|z-9k|^2.
\end{equation}

We will now show that this implies that, for each $m \geq 0$,
\begin{equation}
\label{eq:all errors}
|f(z)-\phi(z)|< \varepsilon_m,\;\text{ for } z \in \overline{V_m'}.
\end{equation}
First we note that, \eqref{eq:all errors} follows from~\eqref{Delta'} and~\eqref{G'} together with~\eqref{approxnew} and~\eqref{approx2} when $m = \ell_n$ or $m = \ell_n-1$, for some $n \geq 0$. {\red Other values of~$m$} are of the form $m = \ell_n + k + 1$, for some $n \geq 0$, $0 \leq k \leq n$,  and it follows from~\eqref{V'} that, in this case,
\[
V'_{m} \subset D(9k, 6\alpha_{n+1}) \subset D_{k}.\]

Therefore, by~\eqref{fDk}, \eqref{eps} and using the fact that $\alpha_{k+1} \leq \alpha_{k}/6$, we have, for $z \in V_{\ell_n + k+1}$, $n \geq 0$ and $0 \leq k \leq n$,
\begin{eqnarray*}
|f(z)-\phi(z)| & \leq & \frac{\varepsilon_{\ell_{k} + k +1}}{\alpha^2_{k}}(6\alpha_{n+1})^2\\
& = & \frac{\alpha_{k+1}^2}{2^{k+2} \alpha^2_{k}} 36 \alpha_{n+1}^2\\
& \leq &\frac{\alpha_{n+1}^2}{2^{k+2}} = \varepsilon_{\ell_{n} + k+1}.
\end{eqnarray*}

Thus~\eqref{eq:all errors} holds for all $m \geq 0$.

It now follows from (\ref{propa}), (\ref{propb}), (\ref{eq:error}) and (\ref{eq:all errors}) that, for $m \geq 0$,
\begin{equation}\label{propB}
\gamma_{m+1}\;\text{surrounds}\; f(\gamma_m);
\end{equation}
\begin{equation}\label{propA}
f(\Gamma_m)\;\text{surrounds}\;\Gamma_{m+1}.
\end{equation}

We now apply Theorem \ref{WDexist} to the Jordan curves $\gamma_m, \Gamma_m$, $m\geq 0$, the compact curves $L_n$, $n \geq 0$, and the bounded domain $D = D(-4, 1)$, noting that these sets satisfy the required hypotheses by construction and by~\eqref{propB}, \eqref{propA}, {\red \eqref{eq:reefs}} and \eqref{approx3}. Part~(i) of Theorem~\ref{thm:main construction} now follows from Theorem \ref{WDexist}, part~(ii) follows from~\eqref{eq:all errors} together with the upper bounds for the errors that we obtained earlier, and part~(iii) follows from (\ref{approx5}) and (\ref{approx6}).

Next we outline the proof of part~(iv). The fact that $f:U_{\ell_{n+1}} \to U_{\ell_{n+1}+1}$ has degree $d_{n+1}$ follows from the final statement of Theorem \ref{WDexist}, since (\ref{degree-4}), (\ref{eq:error}), and (\ref{approx2}) together imply that $f(\gamma_{\ell_{n+1}})$ and $f(\Gamma_{\ell_{n+1}})$ both wind exactly $d_{n+1}$ times round the disc $D(\zeta_{\ell_{n+1}+1}, \rho_{\ell_{n+1}+1}/2)$; for the details of this argument see the proof of \cite[Theorem 5.3]{BEFRS}. Since $\phi$ is univalent in all other cases, the same argument applies to show that $f:U_m \to U_{m+1}$ is univalent in all other cases.

{\red To complete the proof of Theorem~\ref{thm:main construction}, we note that} the double inequality that compares the hyperbolic distances in $U_{\ell_n}$ between points of two orbits under $f$ with the corresponding hyperbolic distances in the discs $G_n$ follows by applying Lemma \ref{lem:hyp estimate} with
\[
s=1-\tfrac{3}{4} \operatorname{dist}(\phi(\gamma_{\ell_n-1}), \partial G_n), \quad r=r_{\ell_n}\quad \text{and}\quad R=R_{\ell_n},
\]
and noting that $f^{\ell_{n+1} - \ell_n}(\overline{D(\kappa_n,r_{\ell_n})} \subset D(\kappa_{n+1},r_{\ell_{n+1}})$; we omit the details which are similar to those given in the proof of the final statement of \cite[Theorem 5.3]{BEFRS}.

\section{Preliminary results for Theorem \ref{realizable}}
In this section we prove some results which we use in order to construct our examples. In particular, we obtain estimates on the orbits of points in a wandering domain $U$ of a transcendental entire function $f$ obtained by applying Theorem~\ref{thm:main construction} with specific Blaschke products $b_n$. Our first result is used repeatedly in our constructions and gives estimates on the distances between orbits under the function $f$ and under the model function $\phi$.

\begin{lem}\label{orberr}
Let $f$ be a transcendental entire function with a wandering domain $U$ arising from applying Theorem~\ref{thm:main construction} with the Blaschke products $b_n$. Then, using the notation of Theorem~\ref{thm:main construction},

(a) if $z,z' \in U_{\ell_n}$, for some $n \geq 0$, we have
 \[
|f^{n+3}(z) - \phi^{n+3}(z')| \leq \alpha_{n+1} + |b_n(z - \kappa_n) - b_n(z'-\kappa_n)|;
\]

(b) and hence, if {\red $z, z' \in  \overline{D(\zeta_0,r_0)}\subset U_0$, we have}
\[
|f^{\ell_{n+1}}(z) - \phi^{\ell_{n+1}}(z')| \leq \alpha_{n+1} + |b_n(f^{\ell_n}(z) - \kappa_n) - b_n(\phi^{\ell_{n}}(z')-\kappa_n)|.
\]
\end{lem}

\begin{proof}

To prove part~(a), we begin by considering the case that $z=z'$. We first use induction to show that, if $z \in U_{\ell_n}$, for some $n \geq 0$, then
\begin{equation}\label{ersum}
|f^m(z) - \phi^m(z)| \leq \sum_{k=0}^{m-1}\varepsilon_{\ell_n +k}, \;\mbox{ for } 1 \leq m \leq n+2.
\end{equation}
We note that, for $m=1$, this holds by Theorem~\ref{thm:main construction} part~(ii). Now assume that~\eqref{ersum} holds for some $m$, $1 \leq m < n+2$. We have
\begin{equation}\label{erm}
|f^{m+1}(z) - \phi^{m+1}(z)| \leq |f(f^m(z)) - \phi(f^m(z))| + |\phi(f^m(z)) - \phi^{m+1}(z)|.
\end{equation}
Since $z \in U_{\ell_n}$, we have $f^{m}(z) \in U_{\ell_n+m} \subset D(\zeta_{\ell_n+m},R_{\ell_n+m})$ and $\phi^m(z) \in D(\zeta_{\ell_n+m},R_{\ell_n+m})$. Also, $\phi$ is a translation on $D(\zeta_{\ell_n+m},R_{\ell_n+m})$ and so, together with {\red Theorem~\ref{thm:main construction}} part~(ii), we can deduce from~\eqref{erm} that
\[
|f^{m+1}(z) - \phi^{m+1}(z)| \leq \varepsilon_{\ell_n +m} + |f^m(z) - \phi^m(z)| \leq \sum_{k=0}^{m}\varepsilon_{\ell_n +k}.
\]
Thus~\eqref{ersum} holds as claimed.

Next, we note that {\red for $z\in U_{\ell_n}$ we have} $f^{n+2}(z), \phi^{n+2}(z) \in \Delta'_n$, on which $\phi$ is a scaling by a factor of $1/\alpha_{n+1}$ followed by a translation and so, by~\eqref{ersum} and Theorem~\ref{thm:main construction} part~(ii),
\begin{eqnarray*}
  |f^{n+3}(z) - \phi^{n+3}(z)| & \leq & |f(f^{n+2}(z)) - \phi(f^{n+2}(z))| + |\phi(f^{n+2}(z)) - \phi^{n+3}(z)|  \\
   &\leq &  \varepsilon_{\ell_n +n+2} + \frac{1}{\alpha_{n+1}} \sum_{k=0}^{n+1}\varepsilon_{\ell_n +k} \\
   &\leq & \frac{\alpha_{n+1}}{2^{n+3}} + {\red \alpha_{n+1}}\sum_{k=0}^{n+1} \frac{1}{2^{k+1}}\\
   & < & {\red\alpha_{n+1}}.
\end{eqnarray*}
This shows that
\begin{equation}\label{estz}
|f^{n+3}(z) - \phi^{n+3}(z)| \leq \alpha_{n+1},
\end{equation}
which is the result of part~(a) in the case that $z=z'$.

We now use this fact to prove {\red part~(a) in general}. If $z,z' \in U_{\ell_n}$, for some $n \geq 0$, then it follows from~\eqref{estz} and the definition of $\phi$ that
\begin{eqnarray*}
  |f^{n+3}(z) - \phi^{n+3}(z')| & \leq & |f^{n+3}(z) - \phi^{n+3}(z))| + |\phi^{n+3}(z) - \phi^{n+3}(z')| \\
   & \leq &   \alpha_{n+1} + |b_n(z - \kappa_n) - b_n(z'-\kappa_n)|.
\end{eqnarray*}
This completes the proof of part~(a).

Now we suppose that {\red $z, z' \in \overline{D(\zeta_0,r_0)}$}. It follows from Theorem~\ref{thm:main construction} part~(i) that {\red $z \in U_0$} and hence $f^{\ell_n}(z) \in U_{\ell_n}$, for $n \geq 0$. It also follows from~\eqref{propa} in the proof of Theorem~\ref{thm:main construction} that {\red $\phi^{\ell_n}(z') \in D(\zeta_{\ell_n},r_{\ell_n})$} and hence, by Theorem~\ref{thm:main construction} part~(i), that $\phi^{\ell_n}(z') \in U_{\ell_n}$, for $n \geq 0$. So part~(b) follows from part~(a) by replacing $z$ and $z'$ by $f^{\ell_n}(z)$ and $\phi^{\ell_n}(z')$ respectively.
\end{proof}

Our next result gives a {\red delicate property of a Blaschke product that appears in one of our examples.} We use this property in the proof of {\red Lemma~\ref{lem:a2} part~(a).}

\begin{lem}\label{lem:b-scaling}
Let $b(z) =\left(\frac{z+1/3}{1+z/3}\right)^2$ and suppose that $0<r<1$. Then $0<r<b(r)<1$ and
\begin{equation}\label{b-ineq}
\left|\frac{b(x)-b(r)}{x-r}\right| < \frac{b^2(r)-b(r)}{b(r)-r}, \quad \text{for } 0<x< b(r).
\end{equation}
\end{lem}
\begin{proof}
Our proof is based on a useful relationship between the cross-ratio of four points $a<b<c<d$ on $\R$, defined as
\[
(a,b,c,d)=\frac{(b-a)(d-c)}{(d-a)(c-b)},
\]
and the Schwarzian derivative of a real function~$f$, defined as
\[
Sf = \frac{f'''}{f''}-\frac32\left(\frac{f''}{f'}\right)^2.
\]
It is well-known that if $f$ is monotonic on an interval $I$ and $Sf<0$ on $I$, then
\begin{equation}\label{Sf-neg}
(f(a),f(b),f(c),f(d))<(a,b,c,d), \quad \text{whenever } a,b,c,d\in I, \; a<b<c<d.
\end{equation}
See, for example, de Melo and van Strien \cite[Section~1]{deMelovanStrien} for a good account of the relationship between functions with negative Schwarzian and the cross-ratio, including a proof of the above fact. Other key properties (also mentioned in \cite{deMelovanStrien}) are that M\"obius maps have zero Schwarzian and the composition rule for Schwarzians is
\[
S(g\circ f)(x) = Sg(f(x))(f'(x))^2+Sf(x).
\]
Since the Schwarzian derivative of a M\"{o}bius map is zero on its domain in $\R$, it follows immediately from this composition rule that the function~$b$ has negative Schwarzian on the interval $(-3,\infty)$.

It is straightforward to check that $1$ is a fixed point of the function~$b$ and that $r<b(r)<1$, for $r \in (0,1)$. Note also that~$b$ is increasing on $(-1/3, \infty)$ and convex on $(-3,1)$. We first prove \eqref{b-ineq} in the case when $r<x<b(r)$, by considering the four points $r,x,b(r),1$. Since $b$ is increasing on $(-1/3,\infty)$ and has negative Schwarzian there, we deduce that
\[
 \frac{(b(x)-b(r))(1-b^2(r))}{(1-b(r))(b^2(r)-b(x))}<\frac{(x-r)(1-b(r))}{(1-r)(b(r)-x)}.
\]
Since $b$ is convex on $(0,1)$ we deduce that
\[
 \frac{1-b(r)}{1-r}<\frac{1-b^2(r)}{1-b(r)}.
\]
We deduce from the previous  two inequalities that
\[
 \frac{b(x)-b(r)}{b^2(r)-b(x)}<\frac{(1-b(r))^2}{(1-r)(1-b^2(r))} \frac{x-r}{b(r)-x}<\frac{x-r}{b(r)-x},
\]
and hence (by taking reciprocals and adding 1 to both sides) that
\[
 \frac{b(x)-b(r)}{b^2(r)-b(r)}<\frac{x-r}{b(r)-r}.
\]
This proves \eqref{b-ineq} in the case when $r<x<b(r)$.

For the case when $0<x<r$, similar reasoning can be used with the points $x,r, b(r), 1$, to deduce that
\[
\frac{b(r)-b(x)}{b^2(r)-b(r)} < \frac{r-x}{b(r)-r}.
\]
This completes the proof of Lemma~\ref{lem:b-scaling}.
\end{proof}

The following lemma describes {\red dynamical} properties of the transcendental entire functions arising from Theorem \ref{thm:main construction} when using specific Blaschke products of a certain form. Two of our examples will be constructed using these Blaschke products. The proof of this result takes several pages.


\begin{lem}
\label{lem:a2}
Let $b(z)=\left(\frac{z+a}{1+az}\right)^2$, where $a \in[1/3,1)$, and let~$f$ be an entire function arising by applying Theorem~\ref{thm:main construction} with $b_n = b$, for $n \geq 0$.
\begin{itemize}
\item[(a)] If $a=1/3$ then there exist {\red $x, y \in U_0\cap \R$}, $N\in\N$ and $c>0$, with $f^n(x)\ne f^n(y)$ for $n\ge 0$, such that
\[
f^{\ell_{N}}(x)=\kappa_N,  \quad\text{and}\quad\kappa_{n}+1-f^{\ell_{n}}(x)\sim \frac{c}{n^{1/2}}\;\text{ as }n\to\infty\]
and
\[
f^{\ell_{N}}(y) = \kappa_N + 1/9, \quad\text{and}\quad\kappa_{n}+1-f^{\ell_{n}}(y) \sim \frac{c}{(n+1)^{1/2}}\;\text{ as }n\to\infty,
\]
and moreover
\[f^{\ell_n}(y)-f^{\ell_n}(x)= \frac{O(1)}{n^{3/2}} \;\text{ as}\;n \to \infty. \]
\item[(b)] If $a=1/2$ then there exist $x, y \in U_0\cap \R$ and $N\in\N$, with $f^n(x)\ne f^n(y)$ for $n\ge 0$, such that
\[
f^{\ell_{N}}(x)= \kappa_n \quad\text{and}\quad\kappa_{n}+1-f^{\ell_{n}}(x) = c \lambda^n(1+\eta_n),\quad\text{for }n\ge N,
\]
and
\[
f^{\ell_{N}}(y)= \kappa_n + 1/4 \quad\text{and}\quad\kappa_{n}+1-f^{\ell_{n}}(y) = c \lambda^{n+1}(1+\xi_n),\quad\text{for } n \ge N,
\]
where $c>0$, $\lambda =2/3$ and $\max\{|\eta_n|,|\xi_n|\}\le 1/10$, for $n \ge N$.
\end{itemize}
\end{lem}
\begin{proof}
First we observe that in both parts it is sufficient to prove the stated results about the behaviours of $f^{\ell_n-\ell_N}(x_N)$ and $f^{\ell_n-\ell_N}(y_N)$ when $x_N,y_N\in U_{\ell_N}$ for some particular positive integer~$N$.

(a) Recall that, by the analysis of the behaviour in~$\D$ of the iterates of~$b$ near its parabolic fixed point~1 (see \cite[Lemma 6.2 (c)]{BEFRS}, for example), there are positive constants~$c$ and~$d$ such that
\begin{equation}\label{1-rn}
1-b^n(0)\sim\frac{c}{n^{1/2}}\;\text{ as } n\to \infty
\end{equation}
and
\begin{equation}\label{rn+1-rn}
b^{n+1}(0)-b^n(0)\sim \frac{d}{n^{3/2}}\;\text{ as } n\to\infty.
\end{equation}
Therefore, we can choose $N$ so large that
\begin{equation}\label{m-large1}
\frac{d}{2n^{3/2}} < b^{n+1}(0)-b^{n}(0)<\frac{2d}{n^{3/2}},\quad \text{for }n\ge N,
\end{equation}
and also such that
\begin{equation}\label{m-large2}
\frac{4}{d\,6^{n}}\le \frac{1}{10},\quad \text{for } n\ge N.
\end{equation}

{\red We then take $r_n=b^{n-N}(0)$, for $n\ge N$, and define
\begin{equation}\label{xn-def}
x_N=\kappa_N\quad\text{and}\quad x_{n+1}=f^{n+3}(x_{n}),\quad\text{for }n\ge N,
\end{equation}
and
\[
x'_n:=\kappa_{n}+r_n\in G_n\cap \R, \quad \text{for } n \geq N.
\]
It follows from the definition of~$\phi$ that
\[
x'_{n+1}=\phi^{n+3}(x'_{n})=b(x'_n-\kappa_n),\quad\text{for }n\ge N.
\]
}
We use Lemma~\ref{lem:b-scaling} to show that the orbit of $x_N=\kappa_N$ under $f$ closely follows that of~$x_N$ under~$\phi$. More precisely, we shall show that
\begin{equation}\label{f-phi-est1}
|x_{n}-x'_{n}|< \frac{1}{10}(r_{n+1}-r_n),\quad\text{for } n \ge N+1.
\end{equation}
Note that it follows from \eqref{xn-def} that $x'_n \in U_{\ell_n}\cap \R$ since $x'_N=x_N\in U_{\ell_N}\cap \R$ and~$f$ is a real entire function.

We shall prove \eqref{f-phi-est1} by using induction to show that
\begin{equation}\label{induction}
|x_n-x'_n|\le (r_{n+1}-r_n)\sum_{k=N+1}^{n}\frac{1}{6^k(r_{k+1}-r_k)},\quad\text{for } n \ge N+1.
\end{equation}
Before proving \eqref{induction}, we show that it implies \eqref{f-phi-est1}. Using~\eqref{m-large1} and \eqref{m-large2}, it follows that, for $n\ge N+1$,
\begin{align*}
\frac{|x_n-x'_n|}{r_{n+1}-r_n} &\le\sum_{k=N+1}^{n}\frac{1}{6^k(r_{k+1}-r_k)}\\
&\le \sum_{k=N+1}^{n}\frac{2(k-N)^{3/2}}{d6^k}\\
&= \frac{2}{d\,6^{N+1}}\sum_{j=0}^{n-N-1}\frac{(j+1)^{3/2}}{6^{j}}\\
&\le \frac{4}{d\,6^{N+1}}<\frac{1}{10},
\end{align*}
since the sum in the penultimate expression is dominated by the geometric series $1+1/2+1/4+\cdots$. Thus \eqref{f-phi-est1} holds.

To start the proof of \eqref{induction}, we have
\[|x_{N+1}-x'_{N+1}|=|f^{N+3}(x_N)-\phi^{N+3}(x_N)|\le \frac{1}{6^{N+1}},\]
by Theorem~\ref{thm:main construction}\,(iv), since $x_N\in U_{\ell_N}$.
Now we assume that \eqref{induction} holds for some~$n\ge N+1$ and deduce that it holds for $n+1$. Note that, whenever \eqref{f-phi-est1} holds (and so whenever \eqref{induction} holds), we have $x'_n \in [\kappa_n,\kappa_n+1)$, by the definition of $x_n$.

By the definition of $\phi$, Lemma~\ref{orberr} part~(a), Lemma~\ref{lem:b-scaling} and Theorem~\ref{thm:main construction}~(iv), we have
\begin{align*}
|x_{n+1}-x'_{n+1}|
&\le \alpha_{n+1}+|b(x_{n}-\kappa_{n})-b(x'_{n}-\kappa_{n})|\\
&\le \frac{1}{6^{n+1}} +|x_{n}-x'_{n}|\left(\frac{b^2(x'_{n}-\kappa_{n})-b(x'_{n}-\kappa_{n})}{b(x'_{n}-\kappa_{n})-(x'_{n}-\kappa_{n})}\right)\\
&= \frac{1}{6^{n+1}}+|x_{n}-x'_{n}|\left(\frac{r_{n+2}-r_{n+1}}{r_{n+1}-r_n}\right)\\
&\le \frac{1}{6^{n+1}}+(r_{n+2}-r_{n+1})\sum_{k=N+1}^{n}\frac{1}{6^k(r_{k+1}-r_k)}\\
&=(r_{n+2}-r_{n+1})\sum_{k=N+1}^{n+1}\frac{1}{6^k(r_{k+1}-r_k)}.
\end{align*}
This proves \eqref{induction}, so \eqref{f-phi-est1} holds.

{\red Next, we define
\[
y_N=\kappa_N+b(0)=\kappa_N+1/9\quad\text{and}\quad y_{n+1}=f^{n+3}(y_{n}),\quad\text{for }n\ge N,
\]
and
\[
y'_n := \kappa_n+ r_{n+1}\in G_n\cap\R,\quad \text{for } n\ge N,
\]
with the same value of~$N$ used earlier. Then
\begin{equation}\label{ynxn}
y'_n-x'_n=r_{n+1}-r_n=b^{n+1-N}(0)-b^{n-N}(0),\;\text{ for } n\ge N,
\end{equation}
and
\[
y'_{n+1}=\phi^{n+3}(y'_{n})=b(y'_n-\kappa_n),\quad\text{for }n\ge N.
\]
}
Reasoning as above we obtain
\begin{equation}\label{f-phi-est2}
|y_{n}-y'_{n}|\le \frac{1}{10}(r_{n+2}-r_{n+1}),\quad\text{for } n \ge N+1.
\end{equation}

Combining \eqref{1-rn} and \eqref{rn+1-rn} with \eqref{f-phi-est1} and \eqref{f-phi-est2}, we obtain
\[
\kappa_{n}+1-x_n\sim \frac{c}{n^{1/2}}\;\text{ as }n\to\infty
\]
and
\begin{eqnarray}
|y_n-x_n| &\leq& |y_n-y'_n|+|y'_n-x'_n|+|x'_n-x_n|\nonumber \\
&\leq& \frac{1}{10}(r_{n+2}-r_{n+1})+(r_{n+1}-r_n)+ \frac{1}{10}(r_{n+1}-r_n) \nonumber \\
&=& \frac{O(1)}{n^{3/2}} \; \text{ as}\;\; n \to \infty, \nonumber
\end{eqnarray}
which gives the required result by taking $x,y \in U_0$ such that $f^{\ell_N}(x) =x_N = \kappa_N$ and $f^{\ell_N}(y)=y_N = \kappa_N + b(0) = \kappa_N + 1/9$. {\red Note that $y_n \ne x_n$ for $n\ge N$, by \eqref{f-phi-est1}, \eqref{ynxn} and \eqref{f-phi-est2}, so we deduce that $f^n(x)\ne f^n(y)$ for $n\ge 0$.}

(b) The proof of part~(b) is similar to that of part~(a), and we outline the argument briefly.

As in part~(a), we take $r_n=b^{n-N}(0)$, for $n\ge N$, for some sufficiently large $N\in \N$ to be specified later in the proof, and put
{\red \begin{equation}\label{xn-def2}
x_N=\kappa_N\quad\text{and}\quad x_{n+1}=f^{n+3}(x_{n}),\quad\text{for }n\ge N,
\end{equation}
and
\[
x'_n:=\kappa_{n}+r_n\in G_{n}\cap\R, \quad \text{for } n \ge N,
\]
so once again
\[
x'_{n+1}=\phi^{n+3}(x'_{n})=b(x'_n-\kappa_n),\quad\text{for }n\ge N.
\]}
Now note that the function~$b$ has fixed point~1 with multiplier $\lambda=2/3$. It follows that
\begin{equation}\label{local-b-est}
1-r_{n+1}=\lambda(1-r_n)(1+O(1-r_n))\;\text{ as } n\to\infty,
\end{equation}
so, for some constant $c>0$,
\begin{equation}\label{lambda-est}
\kappa_{n}+1-x_n=1-r_n \sim c\lambda^n\;\text{ as } n\to\infty.
\end{equation}
Also, since~$b$ is univalent in the disc $\{z:|z-1|<1\}$ (or by a direct calculation), we have
\[
|b'(z)|\leq \lambda(1+C|z-1|),\:\text{ for } |z-1|<1/2,
\]
where $C$ is a positive constant, and so
\begin{equation}\label{b-deriv-est}
|b(w)-b(z)|\le \lambda (1+C|z-1|)|w-z|,\quad\text{for } |w-z|<1/4, |z-1|<1/4.
\end{equation}

As in part~(a), we show that the orbit of $x_N$ under~$f$ closely follows that of $x_N$ under $\phi$. To be precise, we claim that for~$N$ sufficiently large we have
\begin{equation}\label{f-phi-est1-b}
|x_{n}-x'_{n}|\le  \frac{c}{10}\lambda^{n},\;\text{ for } n\ge N.
\end{equation}
Indeed, for $n\ge N$, we have
\begin{align*}
|x_{n+1}-x'_{n+1}|
&\le \alpha_{n+1}+|b(x_{n}-\kappa_{n})-b(x_{n}-\kappa_{n})|\\
&\le \frac{1}{6^{n+1}}+\lambda(1+C(1-r_{n}))|x_{n}-x'_{n}|,
\end{align*}
by Lemma~\ref{orberr} part~(a), Theorem~\ref{thm:main construction}\,(iv), \eqref{lambda-est} and \eqref{b-deriv-est}, provided that~$N$ is sufficiently large. Since $x_N=x'_N=\kappa_N$, it follows easily by induction that, for $n\ge N+1$, we have
\[
\delta_n\le \left(\prod_{k=N+1}^{\infty}(1+C(1-r_k))\right)\left(\sum_{k=N+1}^{n}\frac{1}{(6\lambda)^{k}}\right),\;\text{ where }\delta_n=\frac{|x_{n}-x'_{n}|}{\lambda^n},
\]
and \eqref{f-phi-est1-b} easily follows by \eqref{lambda-est} and by taking~$N$ sufficiently large.

{\red We obtain the first estimate in part~(b) by taking $x \in U_0$ such that $f^{\ell_N}(x) =x_N = \kappa_N$. The second estimate follows by a similar argument but this time we use an orbit under~$f$ whose subsequence passing through $U_{\ell_n}$, $n\ge N$, closely follows the sequence $y'_n := \kappa_n+ r_{n+1}$, $n\ge N$, by taking $y \in U_0$ such that $f^{\ell_N}(y)=y_N = \kappa_N + b(0) = \kappa_N + 1/4$. The proof that $f^n(x)\ne f^n(y)$ for $n\ge 0$ uses \eqref{f-phi-est1-b} and is similar to that in part~(a).}
\end{proof}

Finally {\red in this section}, we give {\red several} estimates for a Blaschke product used in {\red another} of our examples.
\begin{lem}\label{lem:semi}
For $n \geq 0$, let $b_n (z)= \tilde{\mu_n}(\mu_n (z)^2),$ where
\[
\mu_n(z)= \frac{z+s_n}{1+s_nz} \mbox{ and } \tilde{\mu_n}(z)=\frac{z-s_n^2}{1-s_n^2z},
\]
and let
\[
 \lambda_n = \frac{2s_n}{1+s_n^2},
\]
where $s_n \in (0,1)$. Then, for $n\geq 0$,
{\red
\begin{equation}\label{semi1}
\lambda_n\,x \leq \left(\frac{x+ \lambda_n}{1+ \lambda_nx}\right)x= b_n(x) \leq x, \;\mbox{ for }0<x <1,
\end{equation}
and
\begin{equation}\label{semi2}
\lambda_n (y-x) \le b_n(y)-b_n(x)\le \frac{2}{1+\lambda_n}(y-x) , \;\mbox{ for } 0\leq x<y<1.
\end{equation}
}
\end{lem}
\begin{proof}
For $x \in (0,1)$ and $n \geq 0$, we have
\begin{eqnarray*}
b_n(x)&=& \frac {\left(\frac{x+s_n}{1+s_nx}\right)^2-s_n^2}{1- s_n^2\left(\frac{x+s_n}{1+s_nx}\right)^2}  \\
&=& \frac{(x+s_n)^2-s_n^2(1+s_nx)^2}{(1+s_nx)^2-s_n^2(x+s_n)^2} \\
&=& \frac{x^2+2s_nx-2s_n^3x-s_n^4x^2}{1+2s_nx-2s_n^3x-s_n^4}  \\
& = & \left(\frac{(1-s_n^4)x + (1-s_n^2)2s_n}{1-s_n^4 + (1-s_n^2)2s_nx}\right)x\\
& = & \left(\frac{(1+s_n^2)x + 2s_n}{1+s_n^2 + 2s_nx}\right)x\\
&=& \left(\frac{x+ \lambda_n}{1+ \lambda_nx}\right)x.
\end{eqnarray*}
Since
\[a \leq \frac{x+a}{1+ax}\leq 1,\]
for $x,a \in [0,1],$ part~(a) follows.

For part~(b), we deduce from the expression for $b_n$ given in part~(a) that, for $0\leq x<y<1$ and $n\geq 0$,
\begin{eqnarray}
b_n(y)-b_n(x)&=& y \left(\frac{y+\lambda_n}{1+\lambda_ny}\right) -   x \left(\frac{x+\lambda_n}{1+\lambda_nx}\right) \nonumber \\
&=& (y-x) \frac{y+x+ \lambda_n(1+xy)}{(1+\lambda_ny)(1+\lambda_n x)}, \nonumber
\end{eqnarray}
{\red and the conclusion then easily follows from the facts that $0<\lambda_n<1$, for $n \geq 0$, and $0\leq x<y<1$.}
\end{proof}

\section{Proof of Theorem \ref{realizable}}
In this section we construct six examples of bounded oscillating wandering domains, based on the two classifications of simply connected wandering domain given in~\cite{BEFRS}. First, in terms of hyperbolic distances between orbits of points, simply connected wandering domains are classified as follows~(\cite[Theorem A]{BEFRS}.

\begin{thm}[First classification theorem] \label{thm:Theorem A introduction}
Let $U$ be a simply connected wandering domain   of a transcendental entire function $f$ and let $U_n$ be the Fatou component containing $f^n(U)$, for $n \in \N$. Define the countable set of pairs
\[
E=\{(z,z')\in U\times U : f^k(z)=f^k(z') \text{\ for some $k\in\N$}\}.
\]
Then, exactly one of the following holds.
\begin{itemize}
\item[\rm(1)] $\dist_{U_n}(f^n(z), f^n(z'))\lran c(z,z')= 0 $ for all $z,z'\in U$, and we say that $U$ is {\em (hyperbolically) contracting};
\item [\rm(2)] $\dist_{U_n}(f^n(z), f^n(z'))\lran c(z,z') >0$ and $\dist_{U_n}(f^n(z), f^n(z')) \neq c(z,z')$
for all $(z,z')\in (U \times U) \setminus E$, $n \in \N$, and we say that $U$ is {\em  (hyperbolically) semi-contracting}; or
\item[\rm(3)] there exists $N>0$ such that for all $n\geq N$, $\dist_{U_n}(f^n(z), f^n(z')) = c(z,z') >0$ for all $(z,z') \in (U \times U) \setminus E$, and we say that $U$ is {\em (hyperbolically) eventually isometric}.
\end{itemize}
\end{thm}

Next, in terms of convergence of orbits to the boundary there are again three types of simply connected wandering domains (see~\cite[Theorem C]{BEFRS}), though only the latter two are realisable for oscillating wandering domains as explained in the Introduction.

\begin{thm}[Second classification theorem]\label{ThmC}
Let $U$ be a simply connected wandering domain  of a transcendental entire function $f$ and let $U_n$ be the Fatou component containing $f^n(U)$, for $n \in \N$. Then exactly one of the following holds:
\begin{itemize}
\item[\rm(a)]
  $\liminf_{n\to\infty} \operatorname{dist}(f^{n}(z),\partial U_{n})>0$  for all $z\in U$,
    that is, all orbits stay away from the boundary;
\item[\rm(b)] there exists a subsequence $n_k\to \infty$ for which $\operatorname{dist}(f^{n_k}(z),\partial U_{n_k})\to 0$ for all $z\in U$, while for a different subsequence $m_k\to\infty$ we have that
    \[\liminf_{k \to \infty} \operatorname{dist}(f^{m_k}(z),\partial U_{m_k})>0, \quad\text{for }z\in U;\]
\item[\rm(c)] $\operatorname{dist}(f^{n}(z),\partial U_{n})\to 0$ for all $z\in U$, that is, all orbits converge to the boundary.
\end{itemize}
\end{thm}

Each of the examples in this section is constructed by applying Theorem~\ref{thm:main construction} with an appropriate choice of the Blaschke products $b_n$. We make repeated use of the following two results.

\begin{lem}\label{lem51}
Let $f$ be a transcendental entire function with an orbit of wandering domains $(U_n)$ arising from applying Theorem~\ref{thm:main construction} with the Blaschke products $(b_n)_{n \geq 0}$ and suppose that there exist $s,t \in U_0$, $N \in \N$ with
\[
f^{\ell_N}(s), f^{\ell_N}(t) \in \overline{D(\kappa_N, r_{\ell_N})},
\]
where the sequences $(\ell_n)$, $(\kappa_n)$ and  $(r_n)$ are as defined in Theorem~\ref{thm:main construction}.
\begin{itemize}
\item[\rm(a)] If $\dist_{G_n}(f^{\ell_n}(s), f^{\ell_n}(t))\lran  0$ and $f^{\ell_n}(s) \neq f^{\ell_n}(t)$, for $n \geq 0$, then $U_0$ is contracting;

\item[\rm(b)] If $\liminf_{n \to \infty} \dist_{G_n}(f^{\ell_n}(s), f^{\ell_n}(t))>0$ and $f:U_n \to U_{n+1}$ has degree greater than 1 for infinitely many $n \in \N$, then $U_0$ is semi-contracting.

\end{itemize}

\end{lem}
\begin{proof}
(a) In this case it follows from the last part of Theorem~\ref{thm:main construction} that
\[
\dist_{U_n}(f^{\ell_n}(s), f^{\ell_n}(t))\lran  0.
\]
It now follows from Theorem~\ref{thm:Theorem A introduction} that the only possibility is for $U_0$ to be contracting.

(b) In this case it follows from the last part of Theorem~\ref{thm:main construction} that
\[
\liminf_{n \to \infty} \dist_{U_n}(f^{\ell_n}(s), f^{\ell_n}(t)) > 0
\]
and so $U_0$ is not contracting. Since $f:U_n \to U_{n+1}$ has degree greater than 1 for infinitely many $n \in \N$, we know that $U_0$ is not eventually isometric, and so it follows from Theorem~\ref{thm:Theorem A introduction} that $U_0$ is semi-contracting.
\end{proof}

\begin{lem}\label{lem52}
Let $f$ be a transcendental entire function with an orbit of wandering domains $(U_n)$ arising from applying Theorem~\ref{thm:main construction} with the Blaschke products $(b_n)_{n \geq 0}$ and let $s \in U_0 $ with
\[
f^{\ell_n}(s) \in G_n, \; \mbox{ for } n \geq 0,
\]
where the sets $G_n$ and the sequence $(\ell_n)$, $n \geq 0$ are as defined in Theorem~\ref{thm:main construction}.
\begin{itemize}
\item[\rm(a)] {\red If}
 $\liminf_{n \to \infty} \operatorname{dist}(f^{\ell_n}(s),\partial G_{n})>0$, then orbits of points in $U_0$ behave as described in Theorem~\ref{ThmC} part~(b).

\item[\rm(b)] If $\operatorname{dist}(f^{\ell_n}(s),\partial G_{n}) \to 0$ as $n \to \infty$, then orbits of points in $U_0$ behave as described in Theorem~\ref{ThmC} part~(c).

\end{itemize}

\end{lem}

\begin{proof}

We begin by noting that it follows from Theorem~\ref{thm:main construction} that
\begin{equation}\label{Vm}
\phi^m(\Delta_0) = D(\zeta_m, \rho_m) =
\begin{cases}
\Delta_n = D(a_n,\alpha_n), \; \mbox{ if } m=\ell_n -1, \mbox{ where } n \geq 0,\\
G_n = D(\kappa_n,1)\; \mbox{ if } m=\ell_n, \mbox{ where } n \geq 0,\\
D(9k + 4\alpha_{n+1},\alpha_{n+1}) \subset D_k, \; \mbox{ if } m = \ell_n + k+1,  {\red \mbox{ where }  0 \leq k \leq n.}
\end{cases}
\end{equation}
Since we know from Theorem~\ref{thm:main construction} part~(i) that the wandering domains $U_m$ are approximated increasingly well by the sets $\phi_m(\Delta_0)$ as $m \to \infty$, it follows that $\diam U_m \to 0$ as $m \to \infty$ for $m \neq \ell_n$, $n \geq 0$. So, if $s \in U_0$, then
\begin{equation}\label{notln}
\operatorname{dist}(f^{m}(s),\partial U_m) \to 0\;\text{ as } m \to \infty, \; m \neq \ell_n, \; n \geq 0.
\end{equation}

(a) In this case it follows from~\eqref{Vm} together with Theorem~\ref{thm:main construction} part~(i) that
\[
\liminf_{n \to \infty} \operatorname{dist}(f^{\ell_n}(s),\partial U_{\ell_n})>0.
\]
Together with~\eqref{notln}, this implies that orbits of points in {\red $U_0$} behave as described in Theorem~\ref{ThmC} part~(b).

(b) In this case it follows from~\eqref{Vm} together with Theorem~\ref{thm:main construction} part~(i) that
\[
\operatorname{dist}(f^{\ell_n}(s),\partial U_{\ell_n}) \to 0 \;\mbox{ as } n \to \infty.
\]
Together with~\eqref{notln}, this implies that orbits of points in $U_0$ behave as described in Theorem~\ref{ThmC} part~(c).
\end{proof}

%
%
%
%

In some of the examples we make use of the following estimate for the hyperbolic distance in the unit disc.
\begin{obs}\label{obs:hd}
For two points and $r,s \in (0,1)$ with $r<s$ we have that
\[\dist_{\D}(r,s) = \int _{r}^{s} \frac{2dt}{1-t^2},
\]
and so
\begin{equation}\label{hypd}
\log \frac{1-r}{1-s}=  \int _{r}^{s} \frac{dt}{1-t} \leq \dist_{\D}(r,s) = \int _{r}^{s} \frac{2dt}{1-t} =2\log \frac{1-r}{1-s}.
\end{equation}

\end{obs}

We now give the examples that together prove Theorem \ref{realizable}. Examples 1, 2 and 3, which follow, correspond to the three cases of Theorem~\ref{thm:Theorem A introduction}. Within each example we give two functions, corresponding to the two realisable cases of Theorem~\ref{ThmC}.

\begin{ex}[\bf Two contracting wandering domains]
For each of the cases (b) and (c) of Theorem~\ref{ThmC}, there exists a transcendental entire function $f$ having a sequence of bounded, simply connected, oscillating  contracting wandering domains $(U_n)$ with the stated behaviour.

\end{ex}
\begin{proof}

{\bf First example}

We construct an oscillating contracting wandering domain {\red $U_0$} with the behaviour described in Theorem~\ref{ThmC} part~(b) by applying Theorem~\ref{thm:main construction} with  $b_n(z)=z^2$, for $n \geq 0$.

We begin by considering the orbits of points in the disc {\red $D(4,1/12) \subset D(4,r_0)$} under iteration by $\phi$. We note that, if $z \in D(\kappa_n, R_{\ell_n})$, for some $n \geq 0$, then
\[
|\phi^{n+3}(z) - \kappa_{n+1}| = |b_n(z-\kappa_n)| = |z-\kappa_n|^2.
\]
So, if $z \in D(4,1/12)$, then, for $n\ge 0$, we have
\begin{equation}\label{phin}
|\phi^{\ell_{n}}(z) - \kappa_{n}| = |\phi(z) - \kappa_0|^{2^{n}} = |z-7|^{2^{n}}\le (1/12)^{2^n} \to 0, \;\mbox{ as } n\to \infty.
\end{equation}
{\red Next we claim that, if $z \in D(4,1/12)$, then
\begin{equation}\label{ex1}
|f^{\ell_n}(z) - \phi^{\ell_n}(z)| \leq \sum_{i=1}^n \frac{\alpha_i}{3^{n-i}} + \frac{|f(z) - \phi(z)|}{3^n}, \;\mbox{ for } n \geq 0,
\end{equation}
and
\begin{equation}\label{ex0}
|f^{\ell_n}(z) - \kappa_n| \leq \frac14, \;\mbox{ for } n \geq 0.
\end{equation}
}
{\red We prove~\eqref{ex0} and~\eqref{ex1} together using induction.} First, we note that they are true when $n=0$, since {\red if $z\in D(4,1/12)$, then $|f^{\ell_0}(z)-\phi^{\ell_0}(z)| = |f(z)-\phi(z)|\le \eps_0 \le 1/24$, by Theorem~\ref{thm:main construction} part~(ii)}, and $|\phi^{\ell_0}(z)-\kappa_0|=|\phi(z)-\kappa_0|\le 1/12$ by \eqref{phin}.

{\red Next, we suppose that~\eqref{ex1} and~\eqref{ex0} hold} for $n=m \ge 0$. It follows from these two estimates and Lemma~\ref{orberr} part~(b) together with~\eqref{phin}, that if $z\in D(4,1/12)$, then
\begin{eqnarray*}
|f^{\ell_{m+1}}(z) - \phi^{\ell_{m+1}}(z)| & {\red \le} & \alpha_{m+1} + |b_m(f^{\ell_m}(z) - \kappa_m) - b_m(\phi^{\ell_m}(z) - \kappa_m)| \\
& = & \alpha_{m+1} + |(f^{\ell_m}(z) - \kappa_m)^2 - (\phi^{\ell_m}(z) - \kappa_m)^2| \\
&\leq & {\red \alpha_{m+1} + |f^{\ell_m}(z) - \phi^{\ell_m}(z)|\, \left(|f^{\ell_m}(z)-\kappa_m| + |\phi^{\ell_m}(z) - \kappa_m|\right)}\\
& \leq & \alpha_{m+1} + \frac13 |f^{\ell_m}(z) - \phi^{\ell_m}(z)|\\
& \leq & \alpha_{m+1} + \frac13\left( \sum_{i=1}^m \frac{\alpha_i}{3^{m-i}} + \frac{|f(z) - \phi(z)|}{3^m}\right)\\
& = & \sum_{i=1}^{m+1} \frac{\alpha_i}{3^{m+1-i}} + \frac{|f(z) - \phi(z)|}{3^{m+1}},
\end{eqnarray*}
which gives \eqref{ex1} with $n=m+1$, and also
\begin{align*}
|f^{\ell_{m+1}}(z)-\kappa_{m+1}|&\le |f^{\ell_{m+1}}(z) - \phi^{\ell_{m+1}}(z)| + |\phi^{\ell_{m+1}}(z)-\kappa_{m+1}|\\
&\le \sum_{i=1}^{m+1} \frac{\alpha_i}{3^{m+1-i}} + \frac{|f(z) - \phi(z)|}{3^{m+1}} + \left(\frac{1}{12}\right)^{2^{m+1}}\\
&\le \sum_{i=1}^n \frac{1}{3^{m+1-i}6^i} + \frac{1}{12}\,\frac{1}{3^{m+1}} + \left(\frac{1}{12}\right)^{2^{m+1}}\\
&\le \frac{1}{6}+ + \frac{1}{12}\,\frac{1}{3^{m+1}} + \left(\frac{1}{12}\right)^{2^{m+1}}< \frac14,
\end{align*}
which gives \eqref{ex0} with $n=m+1$.

Since {\red $\alpha_n \leq 1/6^n$} for $n \geq 0$, it follows from \eqref{ex1} that if $z \in D(4,1/12)$, then
\[
|f^{\ell_n}(z) - \phi^{\ell_n}(z)| \to 0 \;\mbox{ as } n \to \infty.
\]

Together with~\eqref{phin}, this implies that, if $z \in D(4,1/12)$, then
\begin{equation}\label{fphi}
|f^{\ell_n}(z) - \kappa_n| \to 0 \;\mbox{ as } n \to \infty.
\end{equation}

We now use~\eqref{fphi} together with Lemma~\ref{lem51} and Lemma~\ref{lem52} to show that the wandering domain {\red ${U_0}$} has the required properties. First we take $s,t \in D(4,1/12)$ such that $f(s),f(t) \in D(\kappa_0, r_1)$ with $f^{\ell_n}(s) \neq f^{\ell_n}(t)$, for $n \geq 0$. Since $G_n = D(\kappa_n,1)$, for $n \geq 0$, it follows from~\eqref{fphi} that
\[
\dist_{G_n}(f^{\ell_n}(s), f^{\ell_n}(t)) \to  0 \;\text{ as } n\to \infty,
\]
and hence, by Lemma~\ref{lem51} {\red part~(a)}, $U_0$ is contracting.

Also, it follows from~\eqref{fphi} that
\[
\lim_{n \to \infty} \operatorname{dist}(f^{\ell_n}(s),\partial G_{n})=1
\]
and hence, by Lemma~\ref{lem52} part~(a), orbits of points in $U_0$ behave as described in Theorem~\ref{ThmC} part~(b).

{\bf Second example}
We construct an oscillating contracting wandering domain $U_0$ with the behaviour described in Theorem~\ref{ThmC} part~(c) by applying Theorem~\ref{thm:main construction} with $b_n(z)= \left(\frac{z+1/3}{1+z/3}\right)^2$, for $n \geq 0$. Let $x,y \in U_0$ be as in Lemma \ref{lem:a2} part~(a). Since $G_n = D(\kappa_n,1)$, for $n \geq 0$, it follows from Lemma \ref{lem:a2} part~(a) and {\red the hyperbolic metric estimate}~\eqref{hypd} that
\begin{align*}
\dist_{G_{n}}(f^{\ell_{n}}(x),f^{\ell_{n}}(y))&\le 2\log\frac{\kappa_{n}+1-f^{\ell_{n}}(x)}{\kappa_{n}+1-f^{\ell_{n}}(y)}\\
&= 2\log \left(1+ \frac{f^{\ell_{n}}(t)-f^{\ell_{n}}(x)}{\kappa_n+1-f^{\ell_{n}}(y)}\right)\\
&\sim 2\log\left(1+ \frac{O(1)/n^{3/2}}{c(n+1)^{1/2}} \right)\\
&= \frac{O(1)}{n}\;\text{ as }n\to \infty.
\end{align*}

It now follows from Lemma~\ref{lem51} part~(a) that $U_0$ is contracting. We also know from Lemma~\ref{lem:a2} part~(a) that
\[
\dist(f^{\ell_n}(x),\partial G_n) \to 0 \; \mbox{ as } n \to \infty,
\]
and so it follows from Lemma~\ref{lem52} part~(b) that orbits of points in $U_0$ behave as described in Theorem~\ref{ThmC} part~(c).
\end{proof}

\begin{ex}[\bf Two semi-contracting wandering domains]
For each of the cases (b) and (c) of Theorem~\ref{ThmC}, there exists a transcendental entire function $f$ having a sequence of bounded, simply connected, oscillating  semi-contracting wandering domains $(U_n)$ with the stated behaviour.
\end{ex}
\begin{proof}

{\bf First example}

We construct an oscillating semi-contracting wandering domain with the behaviour described in Theorem~\ref{ThmC} part~(b) by applying Theorem~\ref{thm:main construction} with $b_n (z)= \tilde{\mu_n}(\mu_n (z)^2)$, for $n \geq 0$, where
\[
\mu_n(z)= \frac{z+s_n}{1+s_nz} \; \mbox{ and } \; \tilde{\mu_n}(z)=\frac{z-s_n^2}{1-s_n^2z}.
\]
{\red We shall use the estimates for $b_n$ obtained in Lemma~\ref{lem:semi}, and once again put}
\[
 \lambda_n = \frac{2s_n}{1+s_n^2},\quad\text{for } n\ge 0.
\]
We now choose $s_n \in (0,1)$ with $s_n \to 1$ as $n\to \infty$ {\red so quickly} that
\begin{equation}\label{lambdan}
\prod_{j=0}^{\infty}\lambda_{j} \geq 8/9\; \mbox{ and }\; \prod_{j=0}^{\infty}\frac{2}{1+\lambda_{j}} \leq 4/3.
\end{equation}

We first consider the orbit of the point $4$ under iteration by $f$, noting that
\begin{equation}\label{philn4}
\phi^{\ell_n}(4) = \kappa_n, \;\mbox{ for } n \geq 0.
\end{equation}
{\red It follows from \eqref{philn4},} Lemma~\ref{orberr} part~(b) and \eqref{semi1} in Lemma~\ref{lem:semi} that, for $n \geq 0$,
\begin{eqnarray*}
{\red |f^{\ell_{n+1}}(4) - \kappa_{n+1}|} & = &|f^{\ell_{n+1}}(4) - \phi^{\ell_{n+1}}(4)|\\
&\leq & \alpha_{n+1} + |b_n(f^{\ell_n}(4) - \kappa_n) - b_n(\phi^{\ell_n}(4) - \kappa_n)|\\
& = & \alpha_{n+1} + |b_n(f^{\ell_n}(4) - \kappa_n)|\\
& \leq & \alpha_{n+1} + |f^{\ell_n}(4) - \kappa_n|.
\end{eqnarray*}
Together with Theorem~\ref{thm:main construction} part~(ii) and \eqref{philn4}, this implies that, for $n \geq 0$,
\begin{equation}\label{f4}
|f^{\ell_n}(4) - \kappa_n)| \leq |f(4) - \kappa_0| +  \sum_{i=1}^{n-1} \alpha_i = |f(4) - \phi(4)| +  \sum_{i=1}^{n-1} \alpha_i \leq {\red \frac{1}{24}} + \sum_{i=1}^{n-1} \frac{1}{6^i} < \frac{1}{4}.
\end{equation}

Now we consider the orbit of {\red $19/4=4\tfrac34$} under $f$. {\red Once again, we} begin by considering  the orbit under $\phi$. We claim that
\begin{equation}\label{92}
\phi^{\ell_n}(19/4) - \kappa_n \geq \frac{3}{4} \prod_{i=0}^{n-1} \lambda_i.
\end{equation}
We prove~\eqref{92} by induction first noting that it holds for $n=0$, since $\phi^{\ell_0}(19/4)=\phi(19/4)=\kappa_0 + 3/4$. Next suppose that~\eqref{92} holds for $n=m$. Then, by \eqref{semi1},
\begin{eqnarray*}
 \phi^{\ell_{m+1}}(19/4) - \kappa_{m+1} & = & b_n(\phi^{\ell_{m}}(19/4) - \kappa_{m}) \\
 & \geq & \lambda_m (\phi^{\ell_{m}}(19/4) - \kappa_{m})\\
 & \geq & \frac{3}{4} \lambda_m \prod_{i=0}^{m-1} \lambda_i = \frac{3}{4}\prod_{i=0}^{m} \lambda_i.
\end{eqnarray*}
Thus~\eqref{92} holds for $n=m+1$ and hence, by induction, for all $n \geq 0$. Together with~\eqref{lambdan}, this implies that
\begin{equation}\label{194}
\phi^{\ell_n}(19/4) - \kappa_n \geq \frac{2}{3}, \mbox{ for } n \geq 0.
\end{equation}

Next we claim that, for $n \geq 0$,
\begin{equation}\label{f92}
|f^{\ell_n}(19/4) - \phi^{\ell_n}(19/4)| \leq \sum_{i=1}^{n} \alpha_{i} \prod_{j=i}^{n-1}\frac{2}{1 + \lambda_j} + \prod_{i=0}^{n-1}\frac{2}{1 + \lambda_i} |f(19/4) - \phi(19/4)|.
\end{equation}

We prove~\eqref{f92} by induction noting that it holds for $n=0$. Next suppose that~\eqref{f92} holds for $n=m\ge 0$. Then it follows from Lemma~\ref{orberr} part~(b) and \eqref{semi2} that
\begin{eqnarray*}
 |f^{\ell_{m+1}}(19/4) -  \phi^{\ell_{m+1}}(19/4)| & \leq & \alpha_{m+1} + |b_m(f^{\ell_m}(19/4) - \kappa_m) - b_m(\phi^{\ell_n}(19/4) - \kappa_m)|\\
 & \leq & \alpha_{m+1} + \frac{2}{1 + \lambda_m}|f^{\ell_{m}}(19/4) -  \phi^{\ell_{m}}(19/4)|\\
 & \leq & \alpha_{m+1} + \frac{2}{1 + \lambda_m}\left( \sum_{i=1}^{m} \alpha_{i} \prod_{j=i}^{m-1}\frac{2}{1 + \lambda_j} + \prod_{i=0}^{m-1}\frac{2}{1 + \lambda_i} |f(19/4) - \phi(19/4)|\right)\\
 & = & \sum_{i=1}^{m+1} \alpha_{i} \prod_{j=i}^{m}\frac{2}{1 + \lambda_j} + \prod_{i=0}^{m}\frac{2}{1 + \lambda_i} |f(19/4) - \phi(19/4)|.
\end{eqnarray*}
Thus~\eqref{f92} holds for $n=m+1$ and hence, by induction, for all $n \geq 0$. It now follows from Theorem~\ref{thm:main construction} part~(ii) and~\eqref{lambdan} that
\begin{equation}\label{f92e}
|f^{\ell_n}(19/4) - \phi^{\ell_n}(19/4)| \leq \prod_{j=0}^{\infty}\frac{2}{1 + \lambda_j} \left( {\red \sum_{i=1}^{\infty}\frac{1}{6^i} + \frac{1}{24}} \right)
\le \frac{1}{4} \prod_{j=0}^{\infty}\frac{2}{1 + \lambda_j} \leq \frac{1}{3}.
\end{equation}

It follows from~\eqref{194}, \eqref{f92e} and \eqref{f4} that, for $n \geq 0$,
\begin{eqnarray*}
|f^{\ell_n}(19/4) - f^{\ell_n}(4)| & = & |\phi^{\ell_n}(19/4) - \kappa_n + f^{\ell_n}(19/4) - \phi^{\ell_n}(19/4) + \kappa_n - f^{\ell_n}(4)|\\
& \geq &  |\phi^{\ell_n}(19/4) - \kappa_n| - |f^{\ell_n}(19/4) - \phi^{\ell_n}(19/4)| - |f^{\ell_n}(4) - \kappa_n|\\
& \geq & \frac{2}{3} - \frac{1}{3} - \frac{1}{4} = \frac{1}{12}.
\end{eqnarray*}
Since $G_n = D(\kappa_n,1)$, for $n \geq 0$, together with~\eqref{f4} this implies that
\[
\liminf_{n \to \infty} \dist_{G_n}(f^{\ell_n}(19/4), f^{\ell_n}(4)) > 0.
\]

Also, it follows from Theorem~\ref{thm:main construction} part~(iv) that $f:U_{\ell_n} \to U_{\ell_{n+1}}$ has degree greater than 1, for $n\ge0$. So, by Lemma~\ref{lem51} part~(b), $U_0$ is semi-contracting.

Finally, it follows from~\eqref{f4} that
\[
\liminf_{n \to \infty} \dist(f^{\ell_n}(4), \partial G_n) > 0
\]
and so, by Lemma~\ref{lem52} part~(a), orbits of points in $U_0$ behave as described in Theorem~\ref{ThmC} part~(b).


{\bf Second example}

We construct an oscillating contracting wandering domain $U_0$ with the behaviour described in Theorem~\ref{ThmC} part~(c) by applying Theorem~\ref{thm:main construction} with $b_n(z)=\left(\frac{z+1/2}{1+z/2}\right)^2$, for $n \geq 0$. Let $x,y \in U$ be as in Lemma \ref{lem:a2} part~(b). Since $G_n = D(\kappa_n,1)$, for $n \geq 0$, it follows from Lemma~\ref{lem:a2} part~(b) and {\red the hyperbolic metric estimate}~\eqref{hypd} that {\red for $n$ sufficiently large we have}
\begin{align*}
\dist_{G_{n}}(f^{\ell_{n}}(x),f^{\ell_{n}}(y))&\geq \log\frac{\kappa_{n}+1-f^{\ell_{n}}(x)}{\kappa_{n}+1-f^{\ell_{n}}(y)}\\
&\geq \log \frac{c\lambda^n(1-1/10)}{c\lambda^{n+1}(1+1/10)}\\
&= \log \frac{9}{11\lambda}=\log \frac{27}{22}{\red >0},
\end{align*}
{\red recalling that $\lambda=2/3$}.

Also, it follows from Theorem~\ref{thm:main construction} that $f: U_{\ell_n} \to U_{\ell_n+1}$ has degree greater than 1, for $n\ge 0$, so $U_0$ is semi-contracting by Lemma~\ref{lem51} part~(b). Finally, we know from Lemma~\ref{lem:a2} part~(b) that
\[
\dist(f^{\ell_n}(x),\partial G_n) \to 0 \;\mbox{ as } n \to \infty
\]
and so, by Lemma~\ref{lem52} part~(b), orbits of points in $U_0$ behave as described in Theorem~\ref{ThmC} part~(c).
\end{proof}

\begin{ex}[\bf Two eventually isometric wandering domains]
For each of the cases (b) and (c) of Theorem~\ref{ThmC}, there exists a transcendental entire function $f$ having a sequence of bounded, simply connected, oscillating eventually isometric wandering domains $(U_n)$ with the stated behaviour.
\end{ex}
\begin{proof}

{\bf First example}
We construct an eventually isometric wandering domain $U_0$ with the behaviour described in Theorem~\ref{ThmC} part~(b) by applying Theorem~\ref{thm:main construction} with $b_n(z) = z$, for $n \geq 0$.

Since $b_n$ is univalent, for $n \geq 0$,  it follows from Theorem \ref{thm:main construction} part~(iv) that $f:U_m \to U_{m+1}$ is also univalent, for $m \geq 0$. Thus $U$ is eventually isometric.

We now consider the orbit of $4$ under iteration by $f$. We claim that, for $n \geq 0$,

\begin{equation}\label{4orb}
|f^{\ell_n}(4)-\phi^{\ell_n}(4)| \leq \sum_{i=1}^{n} \alpha_i + |f(4)-\phi(4)|.
\end{equation}

We prove~\eqref{4orb} by induction, noting that it is true for $n=0$. Next, suppose that~\eqref{4orb} holds for $n=m\ge 0$. Then it follows from Lemma~\ref{orberr} part~(b) that
\begin{eqnarray*}
|f^{\ell_{m+1}}(4) - \phi^{\ell_{m+1}}(4)| & = & \alpha_{m+1} + |b_m(f^{\ell_m}(4) - \kappa_m) - b_m(\phi^{\ell_m}(4) - \kappa_m)| \\
& = & \alpha_{m+1} + |f^{\ell_m}(4) - \phi^{\ell_m}(4)| \\
& = & \alpha_{m+1} + \sum_{i=1}^{m} \alpha_i + |f(4)-\phi(4)|\\
& = &  \sum_{i=1}^{m+1} \alpha_i + |f(4)-\phi(4)|.
\end{eqnarray*}
Thus~\eqref{4orb} holds for $n=m+1$ and hence, by induction, for all $n \geq 0$.

Since $\alpha_n \leq 1/6^n$ for $n \geq 0$, and $|f(4) - \phi(4)| \leq 1/24$, by Theorem~\ref{thm:main construction} part~(ii), it follows from~\eqref{4orb} that
\[
|f^{\ell_{n}}(4) - \phi^{\ell_{n}}(4)| \leq 1/2, \;\mbox{ for } n \geq 0.
\]
Since $\phi^{\ell_{n}}(4) = \kappa_n$ and $G_n = D(\kappa_n,1)$, for $n \geq 0$, it follows from Lemma~\ref{lem52} part~(a) that orbits of points in $U_0$ behave as described in Theorem~\ref{ThmC} part~(b).\\

{\bf Second example}

 We construct an eventually isometric wandering domain $U_0$ with the behaviour described in Theorem~\ref{ThmC} part~(c) by applying Theorem~\ref{thm:main construction} with $b_n(z)= b(z) = \frac{z+5/6}{1+5 z/6}$, for $n \geq 0$.

 Since $b_n$ is univalent, for $n \geq 0$, it follows from Theorem \ref{thm:main construction} part~(iv) that $f:U_m \to U_{m+1}$ is also univalent, for $m \geq 0$. Thus $U_0$ is eventually isometric.

 We now consider the orbit of 4 under iteration by $\phi$, noting that
 \[
 \phi^{\ell_n}(4) = \kappa_n + b^n(0), \;\mbox{ for } n\geq 0.
 \]
 The Blaschke product $b$ has an attracting fixed point at 1 and we have
 \begin{equation}
 b^n(0) \to 1 \mbox{ as } n \to \infty \;\mbox{ and }\; b^n(0) \geq 5/6, \mbox{ for } n \in \N,
 \end{equation}
 and so

 \begin{equation}\label{bn0}
 \dist(\phi^{\ell_n}(4),\partial G_n) \to 0 \mbox{ as } n \to \infty, \;\mbox{ and } \; \phi^{\ell_n}(4) - \kappa_n \geq 5/6,\; \mbox{ for } n \in \N.
 \end{equation}
 {\red We also note that} if $0 \le z_1, z_2 <1$, then
 \begin{equation}\label{z1z2}
 |b(z_1) - b(z_2)| = \left| \frac{11(z_1-z_2)}{(6+5z_1)(6+5z_2)}  \right| \leq \frac{11 |z_1-z_2|}{36}.
 \end{equation}

 Next we take a point $x \in D(4,r_0)$ such that $f(x) = \kappa_0$, which is possible by Theorem~\ref{thm:main construction} part~(i), and consider the orbit of~$x$ under iteration by $f$. We claim that, for $n \geq 0$,

\begin{equation}\label{sorb}
|f^{\ell_n}(x)-\phi^{\ell_n}(4)| \leq \sum_{i=1}^{n} \alpha_i \left( \frac{11}{36} \right)^{n-i} \leq \frac{1}{2^n}.
\end{equation}
We prove~\eqref{sorb} by induction. First, we note that it is true if $n=0$, {\red since $f(x)-\phi(4)=0$}. Next, suppose that~\eqref{sorb} holds for $n=m\ge 0$. Then it follows from~\eqref{bn0} and ~\eqref{sorb} that $f^{\ell_{m}}(x) > 0$ and so it follows from Lemma~\ref{orberr} part~(b) together with~\eqref{z1z2} that, {\red for $m \ge 0$}
\begin{eqnarray*}
|f^{\ell_{m+1}}(x) - \phi^{\ell_{m+1}}(4)| & \le & \alpha_{m+1} + |b_m(f^{\ell_m}(x) - \kappa_m) - b_m(\phi^{\ell_m}(4) - \kappa_m)| \\
& \leq & \alpha_{m+1} + \frac{11|f^{\ell_m}(x) - \phi^{\ell_m}(4)| }{36}\\
& \leq & \alpha_{m+1} + \sum_{i=1}^{m} \alpha_i \left( \frac{11}{36} \right)^{m+1-i}\\
& = & \sum_{i=1}^{m+1} \alpha_i \left( \frac{11}{36} \right)^{m+1-i}\\
& \leq & \sum_{i=1}^{m+1} \frac{1}{6^i}\left( \frac{11}{36} \right)^{m+1-i} \leq{\red \frac{m+1}{3^{m+1}} }\leq \frac{1}{2^{m+1}}.
\end{eqnarray*}
Thus~\eqref{sorb} holds for $n=m+1$ and hence, by induction, for all $n \geq 0$.

It follows from~\eqref{sorb} together with~\eqref{bn0} that
\[
\dist(f^{\ell_n}(s),\partial G_n) \to 0 \;\mbox{ as } n \to \infty
\]
and so, by Lemma~\ref{lem52} part~(b), orbits of points in $U_0$ behave as described in Theorem~\ref{ThmC} part~(c).
\end{proof}

\bibliographystyle{alpha}
\bibliography{Oscillating}

\hspace*{-0.4cm} School of Mathematics and Statistics, The Open University, Walton Hall, Milton Keynes  MK7 6AA, UK\\
\textit{Emails:\;}vasiliki.evdoridou@open.ac.uk,\;phil.rippon@open.ac.uk,\;
gwyneth.stallard@open.ac.uk
\end{document}